\documentclass[reqno,twoside]{article}

\usepackage{amssymb}
\usepackage{amsmath}
\usepackage{amsthm}
\usepackage{letterswitharrows}
\usepackage{mathrsfs}
\usepackage{graphicx}
\usepackage{import}
\usepackage{svg}
\usepackage{latexsym}
\usepackage{stmaryrd}
\usepackage{dsfont}
\usepackage{enumitem}
\setlist{topsep=0.3em, itemsep=-0.3em}
\usepackage{moreenum}
\usepackage[bookmarks,hyperfootnotes=false, psdextra=true,hypertexnames=false]{hyperref}
\usepackage{multirow}
\usepackage[dvipsnames]{xcolor}
\usepackage[labelfont=sc,font={stretch=1.2}]{caption}
\colorlet{darkishRed}{red!60!black}
\colorlet{darkishBlue}{blue!60!black}
\colorlet{darkishGreen}{green!50!black}
\colorlet{lightishGreen}{green!70!black}
\hypersetup{
    draft = false,
    bookmarksopen=true,
    colorlinks,
    linkcolor={darkishBlue},
    citecolor={lightishGreen},
    urlcolor={darkishBlue}
}
\usepackage[nameinlink, capitalise, noabbrev]{cleveref}
\usepackage{thmtools}
\usepackage{nameref}
\crefformat{enumi}{#2#1#3}
\crefformat{equation}{#2(#1)#3}
\crefrangeformat{section}{Sections~#3#1#4--#5#2#6}
\crefname{mainresult}{Theorem}{Theorems}

\let\setminus=\smallsetminus
\usepackage{tikz}
\usepackage{tikz-cd}
\usetikzlibrary{calc,through,intersections,arrows, trees, positioning, decorations.pathmorphing, cd, patterns, hobby}
\usepackage{comment}
\usepackage{mathtools}
\usepackage{nccmath}
\usepackage{pifont}
\usepackage[utf8]{inputenc}
\usepackage[T1]{fontenc}
\usepackage{lmodern}
\usepackage[babel]{microtype}
\usepackage[english]{babel}
\usepackage{relsize}

\usepackage[skip=6pt]{subcaption} 
\usepackage{pgfplots}
\usepgfplotslibrary{fillbetween}
\pgfdeclarelayer{ft}
\pgfdeclarelayer{bg}
\pgfsetlayers{bg,main,ft}
\pgfplotsset{compat=1.18}

\usepackage{setspace}

\newcommand{\COMMENT}[1]{{}}

\usepackage{geometry}
\let\setminus=\smallsetminus
\renewcommand{\leq}{\leqslant}
\renewcommand{\geq}{\geqslant}

\renewcommand{\le}{\leq}

\let\rho=\varrho
\let\phi=\varphi

\newcommand{ \N } { \mathbb{N} }

\newcommand{\labtequtag}[3]{%
 \begin{equation} \label{#1} 	\begin{minipage}[c]{0.9\textwidth}  #2 \end{minipage} \ignorespacesafterend \tag{#3} \end{equation} }

\newcommand{\defn}[1]{{\color{darkishRed}{\emph{#1}}}}
\newcommand{\defnm}[1]{{\color{darkishRed}{#1}}}

\makeatletter

\def\calCommandfactory#1{%
   \expandafter\def\csname c#1\endcsname{\mathcal{#1}}}
\def\frakCommandfactory#1{%
   \expandafter\def\csname frak#1\endcsname{\mathfrak{#1}}}

\newcounter{ctr}
\loop
  \stepcounter{ctr}
  \edef\X{\@Alph\c@ctr}
  \expandafter\calCommandfactory\X
  \expandafter\frakCommandfactory\X
  \edef\Y{\@alph\c@ctr}
  \expandafter\frakCommandfactory\Y
\ifnum\thectr<26
\repeat

\setenumerate{label={\normalfont (\roman*)}}

\newtheorem{theorem}{Theorem}[section] 
\newtheorem{corollary}[theorem]{Corollary}
\newtheorem{lemma}[theorem]{Lemma}

\newtheorem{observation}[theorem]{Observation}
\crefname{observation}{Observation}{Observations}
\newtheorem{conjecture}[theorem]{Conjecture}
\crefname{conjecture}{Conjecture}{Conjectures}

\newtheorem{mainresult}{Theorem}
\crefname{mainresult}{Theorem}{Theorems}

\crefname{maincorollary}{Corollary}{Corollaries}

\crefname{mainconjecture}{Conjecture}{Conjectures}

\crefname{mainexample}{Example}{Examples}

\newtheorem{innercustomthm}{}

\newtheorem{innercustomlem}{}

\newenvironment{customthm}[1]
  {\renewcommand\theinnercustomthm{#1}\innercustomthm}
  {\endinnercustomthm}

\newenvironment{customlem}[1]
  {\renewcommand\theinnercustomlem{#1}\innercustomlem}
  {\endinnercustomlem}

\newtheorem{claim}{Claim}
\crefname{claim}{Claim}{Claims}
\AtEndEnvironment{proof}{\setcounter{claim}{0}}

\newenvironment{claimproof}{\noindent\textit{Proof.}}{\hfill\ensuremath{\blacksquare}\medskip}
\usepackage{etoolbox}

\theoremstyle{definition}

\newtheorem{construction}[theorem]{Construction}

\theoremstyle{remark}

\usepackage{etoolbox}
\newbool{pdfBool}
\booltrue{pdfBool} 

\newcommand{\pdfOrNot}[2]{\ifbool{pdfBool}{{#1}}{{#2}}}

\usepackage{svg}

\newbool{arXiv}
\boolfalse{arXiv} 

\newcommand{\arXivOrNot}[2]{\ifbool{arXiv}{{#1}}{{#2}}}

\newbool{TrackBool}
\booltrue{TrackBool} 

\newcommand{\dist}{d}

\newcommand{\td}{tree-decom\-posi\-tion}

\newcounter{mylabelcounter}

\makeatletter
\newcommand{\labelText}[2]{%
#1\refstepcounter{mylabelcounter}%
\immediate\write\@auxout{%
  \string\newlabel{#2}{{1}{\thepage}{{\unexpanded{#1}}}{mylabelcounter.\number\value{mylabelcounter}}{}}%
}%
}

\usepackage[skip=6pt]{subcaption} 
\usepackage{setspace}

\usepackage{pgfplots}
\usetikzlibrary{positioning}
\usetikzlibrary{fillbetween}
\tikzset{every picture/.style={line width=0.8pt}}

\makeatletter
\newcommand\footnoteref[1]{\protected@xdef\@thefnmark{\ref{#1}}\@footnotemark}
\makeatother

\makeatletter
\newcommand{\mylabel}[2]{#2\def\@currentlabel{#2}\label{#1}}
\makeatother

\newcounter{dummy}
\makeatletter
\newcommand\myitem[1][]{\item[#1]\refstepcounter{dummy}\def\@currentlabel{#1}}
\makeatother

\makeatletter
\def\namedlabel#1#2{\begingroup
   \def\@currentlabel{#2}%
   \label{#1}\endgroup
}
\makeatother

\DeclareMathOperator{\RR}{R}

\DeclareMathOperator{\SSS}{\mathbf{SS}}
\DeclareMathOperator{\TTT}{\mathbf{TT}}

\usepackage{csquotes}
\usepackage[backend=biber, style=alphabetic, sorting=nyt, maxnames=99, maxalphanames=99, giveninits=true, sortcites=true]{biblatex}
\makeatletter
\newcommand\thankssymb[1]{\textsuperscript{\@fnsymbol{#1}}}
\makeatother

\usepackage{authblk}

\usepackage{fancyhdr}
\renewcommand\footnotemark{}

\begin{document}

\title{Counterexample to the conjectured coarse grid theorem}

\author{Sandra Albrechtsen$^{*{\includegraphics[height=.5\baselineskip]{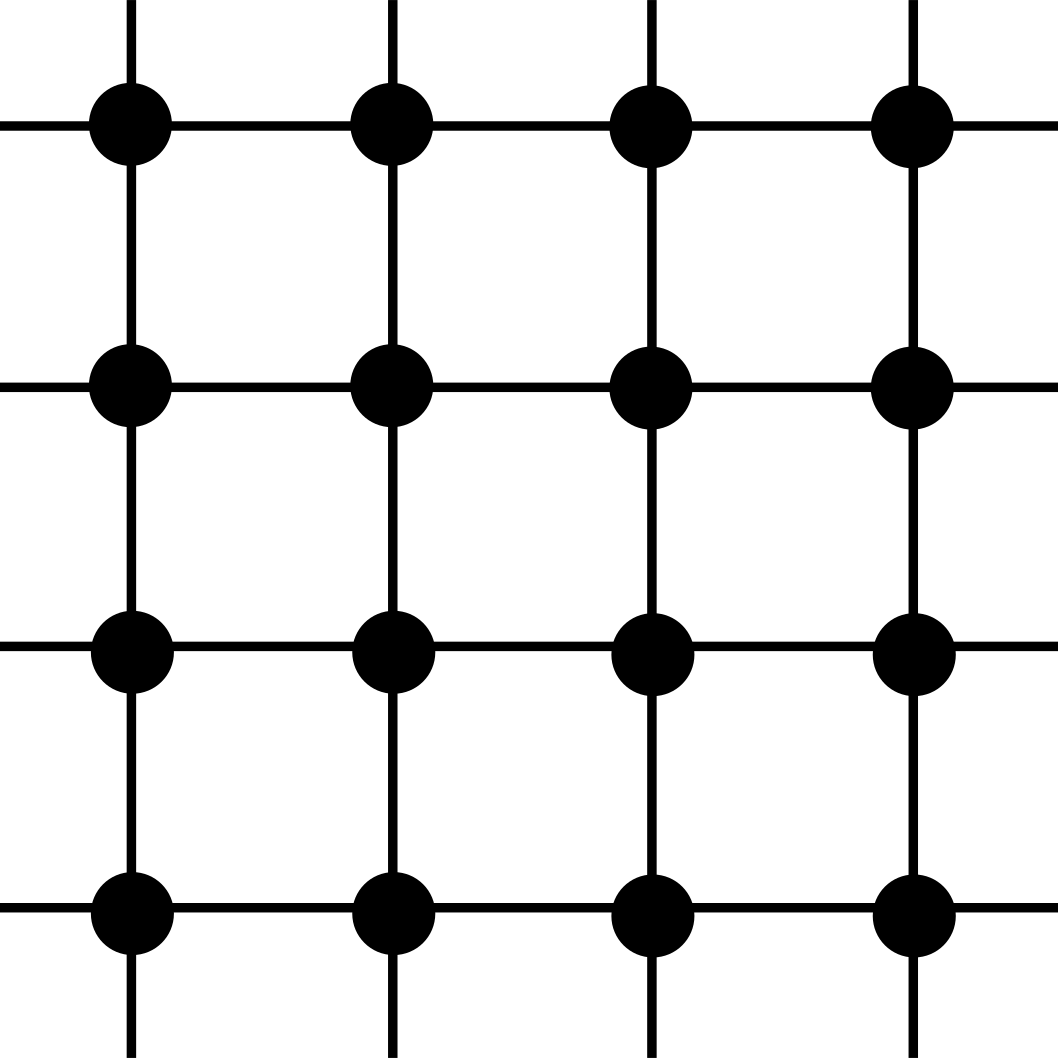}}}$}
\author{James Davies$^{*{\includegraphics[height=.5\baselineskip]{SmallGridHigh.png}}}$}

\affil{${\includegraphics[height=.5\baselineskip]{SmallGridHigh.png}}$ Institute of Mathematics, Leipzig University, Augustusplatz 10, 04109 Leipzig, Germany\\ sandra@albrechtsen-mail.de, jgdavies@uwaterloo.ca}

\date{\vspace{-5ex}}

\thanks{$^*$ Supported by the Alexander von Humboldt Foundation in the framework of the Alexander von Humboldt Professorship of Daniel Král' endowed by the Federal Ministry of Education and Research.}

\maketitle

\begin{abstract}
    We show that for every $M,A,n \in \mathbb{N}$ there exists a graph $G$ that does not contain the $(154\times 154)$-grid as a $3$-fat minor and is not $(M,A)$-quasi-isometric to a graph with no $K_n$ minor. This refutes the conjectured coarse grid theorem by Georgakopoulos and Papasoglu and the weak fat minor conjecture of Davies, Hickingbotham, Illingworth, and McCarty. 
    Our construction is a slight modification of the recent counterexample to the weak coarse Menger conjecture from Nguyen, Scott and Seymour.

    We further modify the construction to show that there are planar graphs that do not have the coarse Erd\H{o}s-P\'{o}sa property.
\end{abstract}



\section{Introduction}

Georgakopoulos and Papasoglu \cite{GP23} recently gave a systematic overview of the emerging area of coarse graph theory, which has received a lot of attention lately.
A key notion in this area is that of fat minors, which was introduced independently by Chepoi, Dragan, Newman, Rabinovich, and Vaxes \cite{FatK23Minor} and by Bonamy, Bousquet, Esperet, Groenland, Liu, Pirot, and Scott \cite{bonamy2023asymptotic}.
Fat minors provide a coarse analogue of graph minors and are inspired by connections to metric geometry.
It is hoped that classical theorems on graph minors could have natural coarse analogues, and this is one of the focuses of coarse graph theory \cite{GP23}.
We delay certain definitions such as that of fat minors, quasi-isometries, and tree-width to \Cref{sec:pre}.

Georgakopoulos and Papasoglu \cite{GP23} conjectured the following coarse analogue of the classical grid theorem of Robertson and Seymour \cite{robertson1986graph}.

\begin{conjecture}[Coarse Grid Theorem] \label{conj:CoarseGrid}
    Let $K,n \in \mathbb{N}$.
    Then, there exist some $M,A,g\in \mathbb{N}$ such that every graph with no $K$-fat $(n\times n)$-grid minor is $(M,A)$-quasi-isometric to a graph of tree-width at most $g$.
\end{conjecture}

There had been a number of encouraging results related to this conjecture~\cite{abrishami2025coarse,albrechtsen2024menger,FatK4,berger2024bounded,chudnovsky2025coarse,GP23,hickingbotham2025graphs,nguyen2025coarse,nguyenasymptotic3}.
For instance, Albrechtsen, Jacobs, Knappe, and Wollan \cite{FatK4} characterized graphs quasi-isometric to graphs of tree-width at most 2 as exactly the graphs forbidding $K_4$ as a fat minor.
Nguyen, Scott, and Seymour \cite{nguyen2025coarse} and Hickingbotham \cite{hickingbotham2025graphs} independently proved that a graph is quasi-isometric to a graph of bounded tree-width if and only if it admits a tree-decomposition in which each bag is contained in the union of a
bounded number of balls of bounded radius.
Nguyen, Scott, and Seymour \cite{nguyenasymptotic3} even proved the coarse analogue of the path-width theorem~\cite{robertson1983graph}, that graphs forbidding some tree as a fat minor are quasi-isometric to graphs of bounded path-width.

On the other hand, there have been some related and now disproven conjectures in coarse graph theory~\cite{SmallCounterexFatMinorConj,CounterexAgelosPanosConjecture,nguyen2025counterexample,NewCounterexToCoarseMenger}. Albrechtsen, Huynh, Jacobs, Knappe, and Wollan \cite{albrechtsen2024menger} and Georgakopoulos and Papasoglu \cite{GP23} independently posed a coarse analouge of Menger's theorem where either there would be $k$ distant paths between vertex sets $X$ and $Y$ in a graph, or $k-1$ bounded radius balls hitting all paths between $X$ and $Y$. While both sets of authors~\cite{albrechtsen2024menger,GP23} proved the case $k=2$, their conjecture was disproved for $k \geq 3$ by Nguyen, Scott, and Seymour \cite{nguyen2025coarse}.
A modification of this construction was used by Davies, Hickingbotham, Illingworth, and McCarty \cite{CounterexAgelosPanosConjecture} to disprove Georgakopoulos and Papasoglu's \cite{GP23} fat minor conjecture that for every graph $H$, the graphs forbidding $H$ as a fat minor are quasi-isometric to graphs forbidding $H$ as a minor.
While their graph $H$ contained two (disjoint) cliques of size $15$, it was noted \cite{CounterexAgelosPanosConjecture} that $H$ can even be chosen to some (large) planar graph.
Albrechtsen, Distel, and Georgakopoulos \cite{SmallCounterexFatMinorConj} recently further showed that the fat minor conjecture is false even for several small graphs including the planar graph $H=K_{2,2,2}$ again using exactly Nguyen, Scott, and Seymour's \cite{nguyen2025coarse} counterexample to the coarse Menger conjecture. 
Still, there was hope for a weak fat minor conjecture where instead of being quasi-isometric to graphs forbidding $H$ as a minor, the graphs are quasi-isometric to graphs forbidding some possibly much larger graph $H'$ as a minor \cite{CounterexAgelosPanosConjecture}.

Nguyen, Scott, and Seymour \cite{NewCounterexToCoarseMenger} recently disproved the weak coarse Menger conjecture \cite{nguyen2025counterexample} (where either there would be $k$ distant paths or $f(k)$ bounded radius balls hitting all such paths) with a significant extension of their previous construction \cite{nguyen2025counterexample}.
We slightly modify their construction further to disprove both the coarse grid theorem of Georgakopoulos and Papasoglu \cite{GP23} (\Cref{conj:CoarseGrid}) in a strong sense (note that every graph containing $K_n$ as a minor has tree-width at least $n-1$), and also the weak fat minor conjecture of Davies, Hickingbotham, Illingworth, and McCarty \cite{CounterexAgelosPanosConjecture}.

\begin{mainresult} \label{main:Counterexample}
    For every $M, A, n \in \N$ with $M \geq 1$, there exists a graph $G$ with the following properties:
    \begin{enumerate}
        \item \label{itm:Counterex:NoGrid} The $(154 \times 154)$-grid is not a $3$-fat minor of $G$.
        \item \label{itm:Counterex:Kn} $G$ is not $(M,A)$-quasi-isometric to a graph with no $K_n$ minor.
    \end{enumerate}
\end{mainresult}

We remark that we focus on the case that the grid is forbidden as a 3-fat minor since if a graph forbids $H$ as a $K$-fat minor, then it is quasi-isometric to a graph forbidding $H$ as a 3-fat minor \cite{CounterexAgelosPanosConjecture}.
We also optimized the proof for simplicity rather than for the size of the forbidden grid fat minor.
\medskip

The first application of Robertson and Seymour's grid theorem \cite{robertson1986graph} was to show that planar graphs have the Erd\H{o}s-P\'osa property, generalizing Erd\H{o}s and P\'osa's~\cite{posa1965independent} classical theorem for cycles. In other words, for every planar graph $H$, there is a function $f:\mathbb{N}\to \mathbb{N}$ such that a graph $G$ either contains $k$ disjoint copies of $H$ as a minor, or there exists some $X\subseteq V(G)$ with $|X|\le f(k)$ such that $G-X$ contains no $H$ minor. Planar graphs are the only graphs with the Erd\H{o}s-P\'osa property.

Georgakopoulos and Papasoglu \cite{GP23} introduced a coarse analogue of the Erd\H{o}s-P\'osa property and conjectured that cycles have the coarse Erd\H{o}s-P\'osa property.
As in \cite{robertson1986graph}, the conjectured coarse grid theorem \cite{GP23} would have straightforwardly implied the coarse analogue of Robertson and Seymour's \cite{robertson1986graph} theorem that a graph $H$ has the coarse Erd\H{o}s-P\'osa property if and only if $H$ is planar.
By further modifying the construction in \Cref{main:Counterexample}, we show that there are planar graphs (see \cref{fig:ErdosPosa:X}) that do not have the coarse Erd\H{o}s-P\'osa property.

\begin{mainresult} \label{main:ErdosPosaPlanar}
    There exists a planar graph $X$ such that for every $K,D,n \in \N$ there is a graph $G$ with the following properties:
    \begin{enumerate}[label=\rm{(\roman*)}]
        \item For every set $Z$ of at most $n$ vertices of $G$, there is a $K$-fat minor-model of $X$ in $G$ that avoids the ball~$B_{G}(Z,D)$ of radius $D$ around $Z$ in $G$.
        \item $G$ does not contain three $3$-fat minor-models of $X$ that are pairwise at distance at least $3$ from each other.
    \end{enumerate}
\end{mainresult}

One major problem on fat minors that remains open is whether classes forbidding a fat minor have bounded asymptotic dimension \cite{bonamy2023asymptotic}. For forbidding a grid as a fat minor, Georgakopoulos and Papasoglu \cite{GP23} conjecture a weakening of \Cref{conj:CoarseGrid} that such graphs have asymptotic dimension at most 1.
There are still several other weakening of \Cref{conj:CoarseGrid} that remain open, and we will discuss these further in \cref{sec:con}.
We will also discuss other possible variants of the coarse Erd\H{o}s-P\'osa property in \cref{sec:con}.

\section{Preliminaries}\label{sec:pre}

Our notions mainly follow~\cite{Bibel}; in particular, a set $U$ of vertices of a graph $G$ is \defn{connected} if $G[U]$ is connected.
We let $\defnm{[n]} := \{1, \dots, n\}$ for every $n \in \N$.

\subsection{Distances}

Let $G$ be a graph.
We write~\defn{$d_G(u, v)$} for the distance of the two vertices~$u$ and~$v$ in~$G$. 
For two sets~$U$ and~$U'$ of vertices of~$G$, we write~\defn{$d_G(U, U')$} for the minimum distance of two elements of~$U$ and~$U'$, respectively.
If one of~$U$ or~$U'$ is just a singleton, then we may omit the braces, writing $\dist_G(v, U') := \dist_G(\{v\}, U')$ for $v \in V(G)$.

Given a set~$U$ of vertices of~$G$, the \defn{ball (in~$G$) around~$U$ of radius $r \in \N$}, denoted as~\defn{$B_G(U, r)$}, is the set of all vertices in~$G$ of distance at most~$r$ from~$U$ in~$G$.
If~$U = \{v\}$ for some~$v \in V(G)$, then we omit the braces, writing~$B_G(v, r)$ for the ball (in $G$) around~$v$ of radius~$r$.

If $Y$ is a subgraph of $G$, then we abbreviate $\dist_G(U,V(Y))$ and $B_G(V(Y),r)$ as $d_G(U,Y)$, and $B_G(Y,r)$, respectively.

\subsection{Fat minors}\label{subsec:fatminors}

Let $G, X$ be graphs.
A \defn{model} $(\cU,\cE)$ of~$X$ in~$G$ is a collection $\cU$ of disjoint, connected sets $U_x \subseteq V(G)$ for vertices $x$ of~$X$, and a collection $\cE$ of internally disjoint $U_{x}$--$U_{y}$ paths~$E_{e}$ for edges $e=xy$ of~$X$ which are disjoint from every~$U_z$ with $z \neq x, y$. 
The $U_x$ are its \defn{branch sets} and the $E_e$ are its \defn{branch paths}.
A model $(\cU, \cE)$ of $X$ in $G$ is \defn{$K$-fat} for $K \in \N$ if $\dist_G(Y,Z) \geq K$ for every two distinct $Y,Z \in \cU \cup \cE$ unless $Y = E_e$ and $Z = U_x$ for some vertex $x \in V(X)$ incident to $e \in E(X)$, or vice versa.
Then $X$ is a \defn{($K$-fat) minor} of $G$ if $G$ contains a ($K$-fat) model of $X$.

\subsection{Quasi-isometries} \label{subsec:QI}

For~$M, A \in \N$ with $M \geq 1$, an \defn{$(M,A)$-quasi-isometry} from a graph~$G$ to a graph~$H$ is a map~$f\colon V(G) \to V(H)$ such that

\begin{enumerate}[label=(Q\arabic*)]
    \item \label{quasiisom:1} $M^{-1} \cdot \dist_G(u, v) - A \leq \dist_H(f(u),f(v)) \leq M \cdot \dist_G(u,v) + A$ for every~$u, v \in V(G)$, and
    \item \label{quasiisom:2} for every vertex $h$ of $H$, there exists a vertex $v$ of $G$ with $\dist_H(h,f(v)) \leq A$.
\end{enumerate}
We say that $G$ is \defn{$(M,A)$-quasi-isometric} to $H$ if there exists an $(M,A)$-quasi-isometry from $G$ to $H$.

The following fact about $(M,A)$-quasi-isometries is immediate from the definition, since for any edge~$uv$ of~$G$ it holds by \ref{quasiisom:1} that $d_H(f(u), f(v)) \leq M+A$. 

\begin{lemma} \label{lem:QIPreservesConn}
    Let $M \geq 1$, $A \geq 0$, and let $f$ be an $(M,A)$-quasi-isometry from a graph $G$ to a graph $H$. If $U \subseteq V(G)$ is connected, then $B_H(f(U), M+A)$ is connected. \qed
\end{lemma}

\cref{lem:QIPreservesConn} has the following converse:

\begin{lemma} \label{lem:QIPreservesConn:Converse}
    Let $M \geq 1$, $A \geq 0$, and let $f$ be an $(M,A)$-quasi-isometry from a graph $G$ to a graph $H$. Let $W \subseteq V(H)$ be connected, and for every $w \in W$, let $u_w$ be a vertex of $G$ with $d_H(f(u_w), w) \leq A$. Then $B_G(U, M(3A+1))$ is connected, where $U := \{u_w : w \in W\}$.
\end{lemma}

\noindent Note that such $u_w$ exists for every $w \in W$ by \ref{quasiisom:2}.

\begin{proof}
    Let $w, w' \in W$ such that $ww' \in E(H)$. Set $u := u_w$ and $u' := u_{w'}$. Then 
    \[
    d_G(u, u') \leq Md_H(f(u), f(u')) + MA \leq M(A+d_H(w,w')+A) + MA = M(2A+1) + MA = M(3A+1).
    \]
    Since $W$ is connected, it follows that $B_G(U, M(3A+1))$ is connected.
\end{proof}

The next lemma essentially says that quasi-isometries preserve separators (in a coarse sense).

\begin{lemma} \label{lem:QIPreservesSeps}
    Let $M \geq 1$, $A \geq 0$ and let $f$ be an $(M,A)$-quasi-isometry from a graph $G$ to a graph $H$. If $X \subseteq V(G)$ separates $Y \subseteq V(G)$ and $Z \subseteq V(G)$ in $G$, then $B_H(f(X), r)$ separates $f(Y)$ and $f(Z)$ in $H$ where $r := r(M,A) = M(M(3A+1)+2M^{-1}A+1)$.
\end{lemma}

\begin{proof}
    Suppose for a contradiction that $B_H(f(X), r)$ does not separate $f(Y)$ and $f(Z)$ in $H$. Then there is an $f(Y)$--$f(Z)$ path $P$ in $H-B_H(f(X), r)$ with endvertices $y' \in f(Y)$ and $z' \in f(Z)$. 
    We pick for every $p \in V(P)$ some $q \in V(G)$ so that $d_H(p, f(q)) \leq A$, and let $U$ be the union of these vertices $q$. In particular, we pick for $y'$ some $y \in Y$ such that $y'=f(y)$ and for $z'$ some $z \in Z$ such that $z'=f(z)$.
    Then $B_G(U, M(3A+1))$ is connected by \cref{lem:QIPreservesConn:Converse}, so there is a $y$--$z$ path $Q$ in $G[B_G(U, M(3A+1))]$. Then 
    \begin{align*}
    d_G(Q,X) &\geq d_G(U, X) - M(3A+1)\\
    &\geq M^{-1}d_H(f(U), f(X)) - M^{-1}A - M(3A+1)\\
    &\geq M^{-1}(d_H(P, f(X))-A) - M^{-1}A-M(3A+1) \\
    &\geq M^{-1}r - 2M^{-1}A-M(3A+1) = 1
    \end{align*}
    and hence $Q$ avoids $X$. This contradicts that $X$ separates $Y$ and $Z$ in $G$.
\end{proof}

\subsection{Separations and tree-decompositions}

A \defn{separation} of a graph~$G$ is a pair $\{A,B\}$ of subsets $A,B$ of $V(G)$ such that $A \cup B = V(G)$ and there is no edge in $G$ between $A\setminus B$ and $B\setminus A$. We call $A, B$ the \defn{sides} of $\{A,B\}$.

A \defn{tree-decomposition} of a graph~$G$ is a pair $(T, \cV)$ that consists of a tree~$T$ and a family $\cV = (V_x)_{x \in V(T)}$ of subsets of~$V(G)$ indexed by the nodes of~$T$ such that:

\begin{enumerate}[label=(T\arabic*)]
    \item \label{TD1} $G = \bigcup_{x \in V(T)} G[V_x]$,
    \item \label{TD2} for every vertex $v$ of $G$, the subgraph of $T$ induced by $\{x \in V(T) : v \in V_x\}$ is connected. 
\end{enumerate}

\noindent Whenever a \td\ is introduced as $(T, \cV)$, we tacitly assume that $\cV = (V_x)_{x \in V(T)}$.

The sets $V_x$ are the \defn{bags} of the \td, and the sets $\defnm{V_e} := V_{x} \cap V_{y}$ for edges $e = xy$ of $T$ are the \defn{adhesion sets} of~$(T, \cV)$.
A \td\ has \defn{width} less than $k \in \N$ if all its bags have size at most~$k$, and a graph~$G$ has \defn{tree-width} at most~$k$ if it has a \td\ of width at most $k$.

In a \td\ $(T, \cV)$ of a graph~$G$ every edge~$e = xy$ of~$T$ induces a separation of~$G$ as follows.
Write~$T_x$ for the component of~$T - e$ containing~$x$ and~$T_y$ for the one containing~$y$.
Then~$\{A^x_e, A^y_e\}$ is a separation of~$G$ where $A^z_e := \bigcup_{u \in T_z} V_z$ for $z \in \{x,y\}$ \cite[Lemma 12.3.1]{Bibel}. We say that~$\{A_e^x, A_e^y\}$ is \defn{induced} by~$e$.
Note that $A^x_e \cap A^y_e = V_e$.

\section{Construction of the counterexample} \label{sec:Construction}

For the proof of \cref{main:Counterexample} we have to define some graph $G$ for every $M, A, n \in \N$. This graph $G$ will essentially be the graph used by Nguyen, Scott and Seymour \cite{NewCounterexToCoarseMenger} to disprove the weak coarse Menger conjecture, except for a small modification of not including certain paths in their construction. 

\begin{figure}[ht]
    \centering
    \begin{tikzpicture}[scale=0.49,auto=left]

\tikzstyle{every node}=[inner sep=1.5pt, fill=black,circle,draw]

\node[blue, ultra thick] (v33) at (33,0) {};
\node[blue,ultra thick] (v2) at (2,0) {};
\draw[dotted, thick] (v2)--(v33);
\node[red, ultra thick] (v3) at (3,0) {};
\node (v4) at (4,0) {};
\node (v5) at (5,0) {};
\node (v6) at (6,0) {};
\node (v7) at (7,0) {};
\node (v8) at (8,0) {};
\node (v9) at (9,0) {};
\node (v10) at (10,0) {};
\node (v11) at (11,0) {};
\node (v12) at (12,0) {};
\node (v13) at (13,0) {};
\node (v14) at (14,0) {};
\node (v15) at (15,0) {};
\node (v16) at (16,0) {};
\node (v17) at (17,0) {};
\node (v18) at (18,0) {};
\node (v19) at (19,0) {};
\node (v20) at (20,0) {};
\node (v21) at (21,0) {};
\node (v22) at (22,0) {};
\node (v23) at (23,0) {};
\node (v24) at (24,0) {};
\node (v25) at (25,0) {};
\node (v26) at (26,0) {};
\node (v27) at (27,0) {};
\node (v28) at (28,0) {};
\node (v29) at (29,0) {};
\node (v30) at (30,0) {};
\node (v31) at (31,0) {};
\node (v32) at (32,0) {};

\node[blue,ultra thick] (u2) at (2.5,1) {};
\node (u4) at (4.5,1) {};
\node (u6) at (6.5,1) {};
\node (u8) at (8.5,1) {};
\node (u10) at (10.5,1) {};
\node (u12) at (12.5,1) {};
\node (u14) at (14.5,1) {};
\node (u16) at (16.5,1) {};
\node (u18) at (18.5,1) {};
\node (u20) at (20.5,1) {};
\node (u22) at (22.5,1) {};
\node (u24) at (24.5,1) {};
\node (u26) at (26.5,1) {};
\node (u28) at (28.5,1) {};
\node (u30) at (30.5,1) {};
\node[blue,ultra thick] (u32) at (32.5,1) {};

\draw[dotted,thick] (u4) to [bend right=20] (v3);
\draw[dotted,thick] (u6) to [bend right=20] (v5);
\draw[dotted,thick] (u8) to [bend right=20] (v7);
\draw[dotted,thick] (u10) to [bend right=20] (v9);
\draw[dotted,thick] (u12) to [bend right=20] (v11);
\draw[dotted,thick] (u14) to [bend right=20] (v13);
\draw[dotted,thick] (u16) to [bend right=20] (v15);
\draw[dotted,thick] (u18) to [bend right=20] (v17);
\draw[dotted,thick] (u20) to [bend right=20] (v19);
\draw[dotted,thick] (u22) to [bend right=20] (v21);
\draw[dotted,thick] (u24) to [bend right=20] (v23);
\draw[dotted,thick] (u26) to [bend right=20] (v25);
\draw[dotted,thick] (u28) to [bend right=20] (v27);
\draw[dotted,thick] (u30) to [bend right=20] (v29);
\draw[dotted,thick] (u32) to [bend right=20] (v31);

\draw[dotted,thick] (u2) to [bend left=20] (v4);
\draw[dotted,thick] (u4) to [bend left=20] (v6);
\draw[dotted,thick] (u6) to [bend left=20] (v8);
\draw[dotted,thick] (u8) to [bend left=20] (v10);
\draw[dotted,thick] (u10) to [bend left=20] (v12);
\draw[dotted,thick] (u12) to [bend left=20] (v14);
\draw[dotted,thick] (u14) to [bend left=20] (v16);
\draw[dotted,thick] (u16) to [bend left=20] (v18);
\draw[dotted,thick] (u18) to [bend left=20] (v20);
\draw[dotted,thick] (u20) to [bend left=20] (v22);
\draw[dotted,thick] (u22) to [bend left=20] (v24);
\draw[dotted,thick] (u24) to [bend left=20] (v26);
\draw[dotted,thick] (u26) to [bend left=20] (v28);
\draw[dotted,thick] (u28) to [bend left=20] (v30);
\draw[dotted,thick] (u30) to [bend left=20] (v32);

\node (t3) at (3.5,2) {};
\node (t7) at (7.5,2) {};
\node (t11) at (11.5,2) {};
\node (t15) at (15.5,2) {};
\node (t19) at (19.5,2) {};
\node (t23) at (23.5,2) {};
\node (t27) at (27.5,2) {};
\node (t31) at (31.5,2) {};

\draw (u2) -- (t3)--(u4);
\draw (u6) -- (t7)--(u8);
\draw (u10) -- (t11)--(u12);
\draw (u14) -- (t15)--(u16);
\draw (u18) -- (t19)--(u20);
\draw (u22) -- (t23)--(u24);
\draw (u26) -- (t27)--(u28);
\draw (u30) -- (t31)--(u32);

\node (s5) at (5.5,3) {};
\node (s13) at (13.5,3) {};
\node (s21) at (21.5,3) {};
\node (s29) at (29.5,3) {};

\draw (t3) -- (s5)--(t7);
\draw (t11) -- (s13)--(t15);
\draw (t19) -- (s21)--(t23);
\draw (t27) -- (s29)--(t31);

\node (r9) at (9.5,4) {};
\node (r25) at (25.5,4) {};

\draw (s5) -- (r9)--(s13);
\draw (s21) -- (r25)--(s29);

\node (q17) at (17.5,5) {};
\draw (r9) -- (q17)--(r25);

\tikzstyle{every node}=[]
\draw (17.5,4.2) node []           {$r=\RR(G)$};
\draw[blue,below left] (v2) node []           {$s_1$};
\draw[blue,above left] (u2) node []  {$s_2$};
\draw[blue,below right] (v33) node []           {$t_1$};
\draw[blue,above right] (u32) node []  {$t_2$};
\draw[red,below] (v3) node [] {$V_2^1$};

\end{tikzpicture}
    \vspace{-2em}
    \caption{Depicted is the graph $G_{4,d,2}$. The dotted lines represent paths of length $d+1$.}
    \label{fig:Ghd2}
\end{figure}

For every $h, d, m \in \N$, with $m \geq 2$, we define a graph $G_{h,d,m}$. For $m = 2$ and $h, d \in \N$ the graph \defn{$G_{h,d,2}$} consists of a binary tree of height $h$, rooted in $r$ and denoted by \defn{$B(G_{h,d,2})$}, a path of length $(d+1)(2^{h+1}-1)$, called the \defn{base path}, and for every leaf of the binary tree two paths of length $d+1$, called \defn{spines}, connecting the leaf to the base path as indicated in \cref{fig:Ghd2}, except that we only add one path for the leftmost leaf and only one for the rightmost leaf.
We call the vertices of the base path of degree $1$ or $3$ the \defn{anchors}. We denote by~\defn{$V_i^1$} the singleton that contains the $i$-th anchor of the base path.
Note that there are $2^{h+1}$ anchors, and the subpath of the base path between any two consecutive anchors has length precisely $d+1$.
Let $s_1, s_2$ and $t_1, t_2$ be as in \cref{fig:Ghd2}. We let \defn{$\RR(G_{h,d,2})$} denote the root $r$ of the binary tree, which we also call the \defn{root} of $G_{h,d,2}$. Let $\defnm{\SSS(G_{h,d,2})} := \{s_1,s_2\}$ and $\defnm{\TTT(G_{h,d,2})} := \{t_1,t_2\}$.  

\begin{figure}[ht]
    \centering
    \begin{tikzpicture}[scale=0.49,auto=left]

\definecolor{dgreen}{rgb}{0,0.7,0}

\foreach \x in {2,4,6,8,10,12,14,16,18,20,22,24,26,28,30} {
\path [fill=gray,opacity=0.5]
(\x,0) to [bend right=20] (\x,2) to [bend right=20] (\x+2,2) to [bend right=20] (\x+2,0) to [bend right=20] (\x,0);
\draw[line width=0.7] (\x,0) to [bend right=20] (\x,2);
\draw[line width=0.7] (\x,2) to [bend right=20] (\x+2,2);
\draw[line width=0.7] (\x+2,2) to [bend right=20] (\x+2,0);
\draw[line width=0.7] (\x+2,0) to [bend right=20] (\x,0);
}

\foreach \x in {18,20,22} {
\path [fill=green,opacity=0.5]
(\x,0) to [bend right=20] (\x,2) to [bend right=20] (\x+2,2) to [bend right=20] (\x+2,0) to [bend right=20] (\x,0);
\draw[line width=0.7] (\x,0) to [bend right=20] (\x,2);
\draw[line width=0.7] (\x,2) to [bend right=20] (\x+2,2);
\draw[line width=0.7] (\x+2,2) to [bend right=20] (\x+2,0);
\draw[line width=0.7] (\x+2,0) to [bend right=20] (\x,0);
}

\tikzstyle{every node}=[inner sep=1.5pt, fill=black,circle,draw]

\node[blue,ultra thick] (v2) at (2,0) {};
\node[red] (v4) at (4,0) {};
\node (v6) at (6,0) {};
\node (v8) at (8,0) {};
\node (v10) at (10,0) {};
\node (v12) at (12,0) {};
\node (v14) at (14,0) {};
\node (v16) at (16,0) {};
\node[Orange,ultra thick] (v18) at (18,0) {};
\node[dgreen] (v20) at (20,0) {};
\node[dgreen] (v22) at (22,0) {};
\node[Orange,ultra thick] (v24) at (24,0) {};
\node (v26) at (26,0) {};
\node (v28) at (28,0) {};
\node (v30) at (30,0) {};
\node[blue,ultra thick] (v32) at (32,0) {};

\node[blue,ultra thick] (u2) at (2,2) {};
\node[red] (u4) at (4,2) {};
\node (u6) at (6,2) {};
\node (u8) at (8,2) {};
\node (u10) at (10,2) {};
\node (u12) at (12,2) {};
\node (u14) at (14,2) {};
\node (u16) at (16,2) {};
\node[Orange,ultra thick] (u18) at (18,2) {};
\node[dgreen] (u20) at (20,2) {};
\node[dgreen] (u22) at (22,2) {};
\node[Orange,ultra thick] (u24) at (24,2) {};
\node (u26) at (26,2) {};
\node (u28) at (28,2) {};
\node (u30) at (30,2) {};
\node[blue,ultra thick] (u32) at (32,2) {};

\node[blue, ultra thick] (t3) at (3,4) {};
\node (t7) at (7,4) {};
\node (t11) at (11,4) {};
\node (t15) at (15,4) {};
\node[Orange,ultra thick] (t19) at (19,4) {};
\node[Orange,ultra thick] (t23) at (23,4) {};
\node (t27) at (27,4) {};
\node[blue,ultra thick] (t31) at (31,4) {};

\draw[dotted,line width=1.4] (t3) to [bend left=20] (v6);
\draw[dotted,line width=1.4] (t3) to [bend left=20] (u6);
\draw[dotted,line width=1.4] (t7) to [bend left=20] (v10);
\draw[dotted,line width=1.4] (t7) to [bend left=20] (u10);
\draw[dotted,line width=1.4] (t11) to [bend left=20] (v14);
\draw[dotted,line width=1.4] (t11) to [bend left=20] (u14);
\draw[dotted,line width=1.4] (t15) to [bend left=20] (v18);
\draw[dotted,line width=1.4] (t15) to [bend left=20] (u18);
\draw[dotted,line width=1.4] (t19) to [bend left=20] (v22);
\draw[dotted,line width=1.4] (t19) to [bend left=20] (u22);
\draw[dotted,line width=1.4] (t23) to [bend left=20] (v26);
\draw[dotted,line width=1.4] (t23) to [bend left=20] (u26);
\draw[dotted,line width=1.4] (t27) to [bend left=20] (v30);
\draw[dotted,line width=1.4] (t27) to [bend left=20] (u30);

\draw[dotted,line width=1.4] (t7) to [bend right=20] (v4);
\draw[dotted,line width=1.4] (t7) to [bend right=20] (u4);
\draw[dotted,line width=1.4] (t11) to [bend right=20] (v8);
\draw[dotted,line width=1.4] (t11) to [bend right=20] (u8);
\draw[dotted,line width=1.4] (t15) to [bend right=20] (v12);
\draw[dotted,line width=1.4] (t15) to [bend right=20] (u12);
\draw[dotted,line width=1.4] (t19) to [bend right=20] (v16);
\draw[dotted,line width=1.4] (t19) to [bend right=20] (u16);
\draw[dotted,line width=1.4] (t23) to [bend right=20] (v20);
\draw[dotted,line width=1.4] (t23) to [bend right=20] (u20);
\draw[dotted,line width=1.4] (t27) to [bend right=20] (v24);
\draw[dotted,line width=1.4] (t27) to [bend right=20] (u24);
\draw[dotted,line width=1.4] (t31) to [bend right=20] (v28);
\draw[dotted,line width=1.4] (t31) to [bend right=20] (u28);

\node (s5) at (5,6) {};
\node (s13) at (13,6) {};
\node[dgreen,ultra thick] (s21) at (21,6) {};
\node (s29) at (29,6) {};

\draw (t3) -- (s5)--(t7);
\draw (t11) -- (s13)--(t15);
\draw (t19) -- (s21)--(t23);
\draw (t27) -- (s29)--(t31);

\node (r9) at (9,8) {};
\node (r25) at (25,8) {};

\draw (s5) -- (r9)--(s13);
\draw (s21) -- (r25)--(s29);

\node (q17) at (17,10) {};
\draw (r9) -- (q17)--(r25);

\draw[green,opacity=0.4,line width=4] (t19)--(s21)--(t23);
\draw[green,opacity=0.4,line width=4] (t19) to [bend left=20] (v22);
\draw[green,opacity=0.4,line width=4] (t19) to [bend left=20] (u22);
\draw[green,opacity=0.4,line width=4] (t23) to [bend right=20] (v20);
\draw[green,opacity=0.4,line width=4] (t23) to [bend right=20] (u20);

\tikzstyle{every node}=[]
\draw (17,9.1) node []           {$r=\RR(G)$};
\draw[blue,below left] (v2) node []           {$s_1$};
\draw[blue,above left] (u2) node []  {$s_2$};
\draw[blue,above left] (t3) node [] {$s_3$};
\draw[blue,below right] (v32) node []           {$t_1$};
\draw[blue,above right] (u32) node []  {$t_2$};
\draw[blue,above right] (t31) node [] {$t_3$};

\draw (3,1) node [] {\large $H_1^2$};
\draw[red] (4,-0.65) node [] {$V_2^2$};

\draw[dgreen,above left] (s21) node [] {$x$};
\draw[dgreen] (23.5,5) node [] {$\Delta(x)$}; 
\draw[Orange] (18,-0.65) node [] {$\SSS(\Delta(x))$};
\draw[Orange] (24,-0.65) node [] {$\TTT(\Delta(x))$};

\end{tikzpicture}
    \vspace{-2em}
    \caption{Depicted is the graph $G_{3,d,3}$, except that $r$ should be identified with the roots of the $H_i^2$'s. The dotted lines represent paths of length $d+1$. The green subgraph is $\Delta(x)$.}
    \label{fig:Ghdm}
\end{figure}

Let us now define \defn{$G_{h,d,m}$} for $m > 2$ and $h, d \in \N$ (see \cref{fig:Ghdm}). We start by taking $n: = 2^{h+1}-1$ disjoint copies $H_1^{m-1}, \dots, H_n^{m-1}$ of $G_{h,d,m-1}$. Then for every $i < n$, we identify the vertices in $\TTT(H_i^{m-1})$ with their `corresponding' vertices in $\SSS(H_{i+1}^{m-1})$. We denote this set of identified vertices by \defn{$V_{i+1}^{m-1}$}, and the resulting graph by $H$. Then $H$ is basically a `path of $G_{h,d,m-1}$'s of length $n$', and the $V_i^{m-1}$ are the `gluing points'. We then add a binary tree of height $h$, rooted in $r$, disjoint from $H$ and denoted by \defn{$B(G_{h,d,m})$}, and for every leaf of the binary tree, we add $2m-2$ paths of length~$d+1$, called \defn{spines}, from the leaves of~$B(G_{h,d,m})$ to $H$ as indicated in \cref{fig:Ghdm}. The graph $G_{h,d,m}$ is then obtained by identifying the roots of the~$H^{m-1}_i$ with $r$, and we call $r$ the \defn{root} of $G_{h,d,m}$.
Phrased differently, $G_{h,d,m}$ is obtained from $G_{h,d,2}$ by replacing each subpath of the base path between two anchors by a copy of $G_{h,d,m-1}$, and then identifying the roots.

Let $s_0, \dots, s_m$ and $t_0, \dots t_m$ be as in \cref{fig:Ghdm}, i.e.\ $s_0, \dots, s_{m-1}$ are the vertices in $\SSS(H_1^{m-1})$ and~$s_m$ is the leftmost leaf of $B(G_{h,d,m})$. We let \defn{$\RR(G_{h,d,m})$} denote the root $r$ of the binary tree, and we let $\defnm{\SSS(G_{h,d,m})} := \{s_1, \dots, s_m\}$ and $\defnm{\TTT(G_{h,d,m})} := \{t_1, \dots, t_m\}$. We also define $\defnm{V_1^{m-1}} := \{s_1, \dots, s_{m-1}\}$ and $\defnm{V_{n+1}^{m-1}} := \{t_1, \dots, t_{m-1}\}$. 
We denote by $H^j_i$, for $j \leq m-2$ and $i \leq n^{m-j+1}$, the copies of $G_{h,d,j}$ inside the~$H_i^{m-1}$, enumerated from left to right. Moreover, we denote by \defn{$V_i^{j}$}, for $j \leq m-2$ and $1 \leq i \leq n^{m-j+1}+1$, the sets $V_k^j$ inside the $H^j_\ell$, enumerated from left to right, e.g.\ for $j = m-2$, there are $n^2+1$ $V_i^{m-1}$'s, of which $n(n-1)$ are contained `strictly inside' some $H^{m-1}_k$ and the others are each a subset of some $V^{m-1}_k$.

Then for every $\ell, m \in \N$, the graph $G_{2\ell+2, 2\ell, m}$ is almost an $(\ell, m)$-block from \cite{NewCounterexToCoarseMenger}, except that we never added the dotted paths between vertices in $\SSS(G)$ or in $\TTT(G)$, and similarly, the vertices in the $V_i^{m-1}$ are also not connected by dotted paths.
This modification is necessary since otherwise the construction contains every graph as a fat minor.

For every vertex $x$ of the binary tree at level $\ell$, let $\Delta = \defnm{\Delta(x)}$ be the subgraph of $G_{h,d,m}$ consisting of all vertices of the binary tree~$B(G)$ below~$x$ together with all~$H^{m-1}_i$ that are joined to this subtree of~$B(G)$ via spines, except for the leftmost and rightmost such~$H^{m-1}_i$ (see \cref{fig:Ghdm} for an example). Note that~$\Delta$ is almost isomorphic to $G_{h-\ell, d, m}$, except that its root $x$ was not identified with the roots of the $H_i^{m-1}$ contained in~$\Delta$. We denote by \defn{$\RR(\Delta)$} the root $x$ of $\Delta$, and we let \defn{$\SSS(\Delta)$} and \defn{$\TTT(\Delta)$} be the set of vertices of~$\Delta$ corresponding to $\SSS(G_{h-\ell, d, m})$ and $\TTT(G_{h-\ell, d, m})$, respectively (see \cref{fig:Ghdm}).
\medskip

We need the following two observations about $G_{h,d,m}$, which follow easily from the construction:

\begin{observation} \label{obs:BlackDotsAreFarApart}
    Let $h,d, m \in \N$ with $m \geq 2$. Then any two vertices in $\bigcup_{j \leq m-1} \bigcup_{i \leq 2^{h+1}j} V^{j}_i$ are at least $2d+2$ apart in $G_{h,d,m}$. \qed
\end{observation}

\begin{observation} \label{obs:SubtrianglesSeparateST}
    Let $h,d, m \in \N$ with $m \geq 2$ and $G:=G_{h,d,m}$. For every $L \in \N$ and every vertex $x$ on the $L$-th level of the binary tree $B(G)$, the set $V(\SSS(\Delta(x)) \cup B_G(\RR(G),L)$ separates $\SSS(G)$ from $\TTT(G)$. \qed
\end{observation}

In the next two sections, \cref{sec:NoFatGrids,sec:KnMinors}, we will prove the following two statements about the graphs~$G_{h,d,m}$, which together easily imply \cref{main:Counterexample}:

\begin{lemma} \label{lem:NoFatGrids}
    The $(154 \times 154)$-grid is not a $3$-fat minor of $G_{h,d,m}$ for any $h,d,m \in \N$.
\end{lemma}

We remark that we optimised the proof for simplicity rather than the size of the grid.

\begin{lemma} \label{lem:KnMinors}
    For every $M, A, n \in \N$ with $M \geq 1$, there exist $h, d, m \in \N$ such that $G_{h,d,m}$ is not $(M,A)$-quasi-isometric to a graph with no $K_n$ minor.
\end{lemma}

\begin{proof}[Proof of \cref{main:Counterexample} given \cref{lem:KnMinors,lem:NoFatGrids}]
    Given $M,A,n \in \N$ with $M\geq 1$, we let $G := G_{h,d,m}$ as provided by \cref{lem:KnMinors} for $M,A,n$. Then \ref{itm:Counterex:NoGrid} follows from \cref{lem:NoFatGrids} and \ref{itm:Counterex:Kn} follows from \cref{lem:KnMinors}.
\end{proof}

In the remainder of this section we prove that the graphs~$G_{h,d,m}$ are still counterexamples to the weak coarse Menger conjecture; a fact that we need for the proofs of \cref{lem:KnMinors,lem:NoFatGrids}.
The proofs are similar to those given by Nguyen, Scott and Seymour \cite{NewCounterexToCoarseMenger}, just that slightly more work is required at one step due to the modification to the construction.
In fact, the next lemma is essentially immediate from one of their proofs.

\begin{lemma} \label{lem:NoTwoPaths}
    Let $h, d, m \in \N$ with $d,m \geq 2$. If $P,Q$ are paths in $G := G_{h,d,m}$ between $\SSS(G)$ and $\TTT(G)$, then either one of $P,Q$ contains $\RR(G)$, or $d_G(P,Q) \leq 2$.
\end{lemma}

\begin{proof}
    If $h = 2\ell+2$ and $d = 2\ell$ for some $\ell \in \N$, then $G := G_{h,d,m}$ is a subgraph of an $(\ell, m)$-block $H$ in~\cite{NewCounterexToCoarseMenger}. The assertion then follows immediately from \cite[\nopp 2.1]{NewCounterexToCoarseMenger}, since we can obtain $G$ from $H$ by deleting only dotted paths, and $d_H(P,Q) \leq 2$ will never be witnessed by a dotted path (as they have length $d+1 \geq 3$).
    
    For all other $h,d$, we remark that the proof of \cite[\nopp 2.1]{NewCounterexToCoarseMenger} only uses $d \geq 1$ and $h \geq 0$, so the assertion also follows for such $G_{h,d,m}$.
\end{proof}

\begin{lemma} \label{lem:NoSmallSeparator}
    Let $\ell, h, d, m \in \N$ with $d \geq 2\ell$, $h \geq 2\ell+2$, and $m \geq 2$. For every set $X$ of less than $m$ vertices of $G:=G_{h,d,m}$, there is a path $P$ in $G$ between $\SSS(G)$ and $\TTT(G)$ with $d_G(P, X \cup \{\RR(G)\}) > \ell$.
\end{lemma}

\begin{proof}
    We prove the following statement: 
    \labtequtag{eq:NoSmallSep:IndHyp}
    {\emph{For every $s \in \SSS(G)$ and $t \in \TTT(G)$ such that $s, t \notin B_G(X, d+\ell+1)$, there is an $s$--$t$ path~$P$ in~$G$ with $d_G(P, X \cup \{\RR(G)\}) \geq \ell$.}}
    {$\ast$}
    Since any two vertices in $\SSS(G) \cup \TTT(G)$ have distance at least $3d + 3 > 2d+2\ell+2$ from each other by construction, and hence at least one vertex in $\SSS(G)$ and one vertex in $\TTT(G)$ does not lie in $B_G(X, d+\ell+1)$, this immediately implies, and strengthens the lemma.

    We proceed by induction on $m \in \N$. For $m = 2$ the assertion is straightforward to check.
    Now let $m \geq 3$, and assume that the result holds for $m-1$. Let $X \subseteq V(G)$ be a set of less than $m$ vertices, and let $s, t$ be as in~\eqref{eq:NoSmallSep:IndHyp}.
    Set $A_x := B_G(x, \ell)$ for $x \in X$ and $A := B_G(X, \ell)$, and abbreviate $H_i := H_i^{m-1}$ and $V_i := V^{m-1}_i$.

    \begin{claim} \label{claim:NoXIsAnInternalVertex}
        We may assume that for every path $P$ of $G$ of length $d+1$ in which all internal vertices have degree two in $G$, no internal vertex of $P$ belongs to $X$.
    \end{claim}

    \begin{claimproof}
        Let $p,q$ be the endvertices of $P$. By the definition of $G$, both $p, q$ have degree $3$. If $x \in X$ is an internal vertex of such a path $P$, then, since $P$ has length $d+1$, at most one of $p,q$ belongs to $A_x$, so we may assume that $p \notin A_x$. Let $X' = (X \setminus \{x\}) \cup \{p'\}$. If the result is true for $X'$ then it is also true for $X$, so by repeating this at most $|X|$ times, we may assume there is no such $x$.
    \end{claimproof}

    \begin{claim} \label{claim:NoViMeetsAllAx}
        We may assume that no $V_i$ is contained in $B_{H_i}(A, d+1) \cup B_{H_{i+1}}(A, d+1)$. 
    \end{claim}

    \begin{claimproof}
        Assume that some $V_i$ is contained in $B_{H_i}(A, d+1) \cup B_{H_{i+1}}(A, d+1)$. Since $|V_i|=m-1=|X|$ and because the vertices in $V_i$ have distance at least $3d+3 > 2d+2\ell+2$ from each other in $H_i$ as well as in~$H_{i+1}$, it follows that $H_i \cup H_{i+1}$ meets all $A_x$, which implies that $B(G) - B_G(\RR(G), \ell)$ avoids~$A$.
        Moreover, by \cref{obs:BlackDotsAreFarApart}, $V_{i-2}, V_{i-1}, V_{i+1}$ and $V_{i+2}$ avoid~$A$, which also implies that $H_j-B_G(\RR(G), \ell)$ avoids~$A$ for all $j \neq i-1,i$. 
        By the construction of~$G$, there is a leaf $L$ of the binary tree~$B(G)$ whose spines `jump over'~$V_i$, in that they connect $L$ to either $V_{i-2}$ and $V_{i+1}$ or to $V_{i-1}$ and $V_{i+2}$. By the previous two arguments and by \cref{claim:NoXIsAnInternalVertex}, we can thus use $L$ together with two of its spines to construct a path~$Q$ from $\bigcup_{j < i} H_i$ to $\bigcup_{j > i} H_i$ that avoids $A$. 
        Since also all $H_j-B_G(\RR(G), \ell)$ with $j \neq i-1, i$ avoid $A$, we can easily extend $Q$ to an $s$--$t$ path through the $H_j-B_G(\RR(G), \ell)$, which are clearly connected.
    \end{claimproof}

    By \cref{claim:NoViMeetsAllAx}, we may assume that no $V_i$ is contained in $B_{H_i}(A, d+1) \cup B_{H_{i+1}}(A, d+1)$, so we may pick from every $V_i$ some $v_i$ avoiding this set. By assumption, we may define $v_1 := s$ and $v_{n+1} := t$.

    \begin{claim} \label{claim:SomeHiMeetsAllAx}
        We may assume that some $H_i$ meets all $A_x$.
    \end{claim}

    \begin{claimproof}
        Assume that every $H_i$ meets at most $m-2$ of the $A_x$. 
        For every $x \in X$, the ball $A_x$ contains at most one vertex in $V_i \cup V_{i+1}$, since they pairwise have distance at least $d + 1 \geq 2\ell+1$ by \cref{obs:BlackDotsAreFarApart}. 
        Consequently, $A \cap V(H_i)$ is the union of at most $m-2$ balls of $H_i$ of radius at most $\ell$. Then for every $i \in \N$ we may apply the induction hypothesis \eqref{eq:NoSmallSep:IndHyp} in $H_i$ to $v_i$ and $v_{i+1}$. This yields for every $i \in \N$ an $v_i$--$v_{i+1}$ path in $H_i$ that avoids $A \cup B_G(\RR(G), \ell)$. Clearly, the concatenation of all $P_i$ is the desired $s$--$t$ path.
    \end{claimproof}

    By \cref{claim:SomeHiMeetsAllAx}, we may assume that some $H_i$ meets all $A_x$. Similarly as in the proof of \cref{claim:NoViMeetsAllAx}, we obtain that $B(G)-B_G(\RR(G), \ell)$ avoids $A$, so without loss of generality, we find a path $Q$ from either $v_{i-2}$ to~$v_{i+1}$ or from~$v_{i-1}$ to~$v_{i+2}$ using two suitable spines. 
    Now since every~$A_x$ meets $H_i$ but avoids $v_i$ and $v_{i+1}$, every other $H_j$ meets at most $m-2$ sets $A_x$. Hence, we may use the same argument as in the proof of \cref{claim:SomeHiMeetsAllAx} to find in every $H_j$ some $v_j$--$v_{j+1}$ path $P_j$ that avoids $A \cup B_G(\RR(G), \ell)$. Then $Q$ together with the paths $P_j$ is the desired $s$--$t$ path.
\end{proof}

Note that it follows from \cref{lem:NoTwoPaths,lem:NoSmallSeparator} that the graphs $G_{h,d,m}$ are still a counterexample to the weak coarse Menger conjecture.
For our proof of \cref{main:Counterexample} we will need to use \cref{lem:NoTwoPaths,lem:NoSmallSeparator} separately however.

We conclude the section with the following corollary of \cref{lem:NoTwoPaths}, which we will need for the proofs of \cref{lem:NoFatGrids} and \cref{main:ErdosPosaPlanar}.
Recall that $\SSS(G) = \{s_1, \dots, s_m\}$ (enumerated from `bottom' to `top') (see \cref{fig:Ghdm}).

\begin{corollary} \label{cor:NoTwoPathsBetweenSameAdhesion}
    Let $h,d,m \in \N$ with $d,m \geq 2$. Let $P,Q$ be paths in $G := G_{h,d,m}$ such that $P$ has endvertices $s_{k_1}, s_{\ell_1} \in \SSS(G)$ and $Q$ has endvertices $s_{k_2}, s_{\ell_2} \in \SSS(G)$. If $k_1, k_2 < \ell_1, \ell_2$, then either one of $P,Q$ contains $\RR(G)$, or $d_G(P,Q) \leq 2$.
\end{corollary}

We remark that by the symmetry of $G_{h,d,m}$, \cref{cor:NoTwoPathsBetweenSameAdhesion} is also true with $\TTT(G)$ instead of $\SSS(G)$.

\begin{proof}
    Let paths $P,Q$ be given, and let $j := \min\{\ell_1, \ell_2\}-1$; in particular, $V^{j}_{1} \subseteq V^m_1 = \SSS(G)$ contains $s_{k_1}, s_{k_2}$ but not $s_{\ell_1}, s_{\ell_2}$. Set $V' := V^{j}_{1}$, and let $H' := H^{j}_{1}$ be the (unique) copy of~$G_{h,d,j}$ in~$G$ with $\SSS(H') = V'$ (see \cref{fig:NoFatGrid:Proof}).
    Then every path in $G$ between $V'$ and $\SSS(G) \setminus V'$ that avoids $\RR(G)$ has to meet $\TTT(H')$ by the construction of~$G$ and as can be seen by \cref{fig:NoFatGrid:Proof}.
    Hence, either one of $P, Q$ meets $\RR(G)$, or $P, Q$ both meet $\TTT(H')$. Since we are done in the former case, we may assume the latter.
    But then $P, Q$ contain subpaths $P', Q'$ between $\SSS(H')$ and $\TTT(H')$ that are contained in~$H'$ and that avoid $\RR(H') = \RR(G)$. By \cref{lem:NoTwoPaths} (and because $H'$ is isomorphic to $G_{h,d,j}$), we have $d_{H'}(P', Q') \leq 2$, and hence $d_G(P,Q) \leq 2$, as desired.
     \begin{figure}[ht]
        \centering
        \include{Sketch_Lemma42_new}
        \vspace{-2em}
        \caption{A sketch of the situation in the proof of \cref{cor:NoTwoPathsBetweenSameAdhesion}. Every path in $G$ between $\{s_{k_1},s_{k_2}\}$ and $\{s_{\ell_1},s_{\ell_2}\}$ has to intersect $\TTT(H') \cup \{\RR(G)\}$.}
        \label{fig:NoFatGrid:Proof}
    \end{figure}
\end{proof}

\section{No large fat grid minors} \label{sec:NoFatGrids}

In this section we prove \cref{lem:NoFatGrids}, which we recall here for convenience:

\begin{customlem}{\cref*{lem:NoFatGrids}}
    The $(154 \times 154)$-grid is not a $3$-fat minor of $G_{h,d,m}$ for any $h,d,m \in \N$ with $m \geq 2$.
\end{customlem}

Let $G$ be a graph and $K, n \in \N$. We call a set $W$ of at least $2n$ vertices of $G$ \defn{$K$-fat $n$-path-connected (in~$G$)} if $d_G(u, v) \geq K$ for all $u,v \in W$ and if for every $\ell \leq n$ and for every two subsets $A, B \subseteq U$ of size $\ell$, there are $\ell$ paths in $G$ between $A$ and $B$ that are pairwise at least $K$ apart in $G$. 

It is well known that in a $(2n \times n)$-grid, for $n \in \N$, the vertices of any row are $1$-fat $n$-path-connected. By the definition of fat minors, the following is immediate:

\begin{observation} \label{obs:LinesInGridAreNConn}
    Let $K,n \in \N$, and let $(\cU, \cE)$ be a $K$-fat model of the $(2n \times n)$-grid $H$ in a graph~$G$. Then for every $i \leq n$, every set $W$ consisting of one vertex of each branch set in $\cU$ corresponding to a vertex of the $i$-th row of $H$ is $K$-fat $n$-path-connected. \qed
\end{observation}

The proof of \cref{lem:NoFatGrids} consists of two steps. We first show that no $V_i^j$ in some $G_{h,d,m}$ contains a $3$-fat $7$-path-connected set (see \cref{lem:NoFatLinkedSet:Vi} below).  
To conclude the proof of \cref{lem:NoFatGrids}, we suppose for a contradiction that $G$ contains a $3$-fat model of the $(154\times 154)$-grid, which provides a $3$-fat $77$-path-connected set $W$ in $G$ by \cref{obs:LinesInGridAreNConn}. We then construct a \td\ $(T, \cV)$ of $G_{h,d,m}$, and use $(T, \cV)$ to `trap' a large subset of our hypothetical $3$-fat $77$-path-connected set $W$ inside a bag in $\cV$. But the bags in $\cV$ will essentially consist of the union of four $V_i^j$'s (plus up to eight additional vertices), which by the pigeon hole principle implies that one such $V^j_i$ contains a subset of $W$ of size at least~$14$. But this contradicts \cref{lem:NoFatLinkedSet:Vi}.

\begin{lemma} \label{lem:NoFatLinkedSet:Vi}
    Let $h, d, m \in \N$ with $m \geq 2$. Then no $V_i^j$ in $G_{h,d,m}$ contains a $3$-fat $7$-path-connected set. 
\end{lemma}

\begin{proof}
    Suppose for a contradiction that some $V_i^j$ of $G := G_{h,d,m}$ contains a $3$-fat $7$-path-connected set $W$. Let $j$ be maximal with this property, i.e.\ $V_i^j \not\subseteq V_k^{j+1}$ for any $k$, and abbreviate $V_i := V^j_i$. Let $H_{k-1}^j$ and $H_k^j$ be the two copies of $G_{h,d,j}$ that were glued together along $V_i$, and abbreviate them as $H_{i-1}$ and $H_i$ (in the case where $V_i = V_1^m$ or $V_i = V_{2^{h+1}}^m$, we let $H_{i-1}$ or $H_{i}$, respectively, be the empty graph). Also, let $H$ be the (unique) copy of $G_{h,d,j+1}$ in $G$ that contains $V_i = V_i^j$, i.e.\ if $j \leq m-2$, then $H = H^{j+1}_\ell$ for some $\ell$, and if $j=m$, then $H = G$.
    
    Without loss of generality assume that $|W| = 14$. We enumerate $W = \{w_1, \dots, w_{14}\}$ from `bottom to top'. More precisely, recall that $V_i = \TTT(H_{i-1})$, so every $w \in W$ is some $t_a \in \TTT(H_{i-1})$, which are enumerated by definition. We now enumerate the vertices in $W$ so that $j_a \leq j_b$ whenever $a \leq b$ and $w_a = t_{j_a}$ and $w_b = t_{j_b}$. Let $A := \{w_1, \dots, w_7\}$ and $B := \{w_{8}, \dots, w_{14}\}$. Since $W$ is $3$-fat $7$-path-connected, there are seven $A$--$B$ paths $P_1, \dots, P_7$ in $G$ that are pairwise at least $3$ apart. At most one $P_k$ contains $\RR(G)$ and at most one $P_k$ contains the unique leaf $L$ of the binary tree $B(H)$ that is joined to $V_i$ via spines. Without loss of generality, assume that $P_1, \dots, P_5$ avoid $\RR(G)$ and $L$. 
    
    By the construction of $G$, the sets $V_{i-1} := V^j_{i-1}$ and $V_{i+1} := V^j_{i+1}$ together with $\RR(G)$ and $L$ separate $H_{i-1} \cup H_{i}$ from $G-(H_{i-1} \cup H_i)$. Hence, every $P_k$, with $k \in [5]$, is either contained in $H_{i-1} \cup H_i$ or meets $V_{i-1} \cup V_{i+1}$. We claim that at most one $P_k$ intersects~$V_{i+1}$ and, analogously, at most one~$P_k$ intersects~$V_{i-1}$. Indeed, the $P_k$ start in $V_i = \SSS(H_{i})$ and are pairwise at least $3$-far apart, so if two of them meet $V_{i+1} = \TTT(H_{i})$, then they contain two subpaths between $\SSS(H_{i})$ and $\TTT(H_{i})$ that are contained in $H_{i}$ and that are at least $3$-far apart. But since they avoid $\RR(G) = \RR(H_{i})$, and because $H_{i}$ is isomorphic to $G_{h,d,j-1}$, this contradicts \cref{lem:NoTwoPaths}.  
    Hence, at most one $P_k$ meets $V_{i+1}$ and, analogously, at most one~$P_k$ meets $V_{i-1}$. 

    Without loss of generality assume that $P_1, P_2, P_3$ avoid $V_{i-1} \cup V_{i+1}$. In particular, $P_1, P_2, P_3$ are contained in $H_{i-1} \cup H_i$.
    Let $A' \subseteq A$ and $B'\subseteq B$ comprise the endvertices of $P_1, P_2, P_3$ in $A$ and $B$, respectively.
    Let $A'=:\{a_1,a_2,a_3\}$ and $B'=:\{b_1,b_2,b_3\}$ be again enumerated from `bottom to top' as described above. Let further $C \supseteq A'$ comprise all vertices of~$V_i$ up to $a_3$, and let $D := V_i \setminus C$ comprise all vertices of~$V_i$ above~$a_3$; in particular, $C \cup D = V_i$ and $C\cap D = \emptyset$. Since every~$P_k$, with $k \in [3]$, has one endvertex in~$C$ and one in~$D$, it contains a subpath $Q_k \subseteq P_k$ that has one endvertex in~$C$ and one in $D$ and that is internally disjoint from $C \cup D$. As $V_i$ separates $H_{i-1}$ and $H_i$ in $(H_{i} \cup H_{i-1}) - \RR(G)$, and because $C \cup D = V_i$, each $Q_k$ is contained in either $H_{i-1}$ or in $H_i$. 
    
    By symmetry and the pigeon hole principle, we may assume that $Q_1, Q_2 \subseteq H_i$. Now let us recall that $Q_1, Q_2$ have both endvertices in $V_i$, they are both contained in $H_i$, and they do not contain $\RR(G) = \RR(H_i)$. Moreover, $V_i = \SSS(H_i) = \{s_1, \dots, s_{j-1}\}$ (enumerated from `bottom to top'), and $Q_1, Q_2$ were chosen so that if $s_{k_1}, s_{\ell_1}$ are the endvertices of $Q_1$, and $s_{k_2}, s_{\ell_2}$ are the endvertices of $Q_2$, then $k_1, k_2 \leq \ell_1, \ell_2$. 
    Since $H_i$ is isomorphic to $G_{h,d,j}$, applying \cref{cor:NoTwoPathsBetweenSameAdhesion} to $H_i$ and $Q_1,Q_2$ yields $d_G(Q_1, Q_2) \leq 2$, a contradiction.
\end{proof}

Next, we give the construction of a \td\ of $G_{h,d,m}$. This construction is divided into two steps.
The \td\ from \cref{constr:TreeDecomp1} is a generalisation of the one from \cite{SmallCounterexFatMinorConj} that was considered for the original counterexample to the coarse Menger conjecture \cite{nguyen2025counterexample}.

\begin{construction} \label{constr:TreeDecomp1}
    Let $h,d,m \in \N$ with $m \geq 2$, and let $G:=G_{h,d,m}$. Let $T$ be the tree obtained from the binary tree $T':=B(G)$ in $G$ by adding for every vertex $x$ of $T'$ new vertices $v_x, u^1_x, \dots, u^{2m-2}_x$ and joining them to $x$. For every $x \in V(T')$ with (ordered from left to right) children $x_1,x_2$, we let 
    \begin{equation} \label{eq:Bags}
        V_x := \{\RR(G), x, x_1, x_2\} \cup \SSS(\Delta(x)) \cup \TTT(\Delta(x)) \cup \TTT(\Delta(x_1)) \cup \SSS(\Delta(x_2)),
    \end{equation}
    and $V_{v_x} := V(H_{i_x}^{m-1})$ where $i_x$ is such that $H^{m-1}_{i_x}$ lies directly `below' $x$, as indicated in \cref{fig:TreeDecomp}. Moreover, we let each of the bags of nodes $u^j_x$ consist of precisely one spine between $B(G)$ and $H^{m-1}_{i_x}$.
    It is easy to check that $(T, \cV)$ is a \td\ of $G$. Every adhesion set $V_e$ of an edge $e=xy \in V(T)$ that is not of the form $xv_x$ or $xu_x^i$ and where $y$ is a child of $x$ in $T$ is of the form
    \begin{equation} \label{eq:AdhesionSets}
        V_e = \{\RR(G), y\} \cup \SSS(\Delta(y)) \cup \TTT(\Delta(y)).
    \end{equation}
    Moreover, the adhesion set of the unique edge $e=xv_x$ of $T$ is $V_e = \SSS(H^m_{i_x}) \cup \TTT(H^m_{i_x}) \cup \{\RR(G)\}$.
\end{construction}

\begin{figure}[ht]
    \centering
    \begin{tikzpicture}[scale=0.49,auto=left]

\foreach \x in {2,4,6,8,10,12,14,16,18,20,22,24,26,28,30} {
\path [fill=gray,opacity=0.5]
(\x,0) to [bend right=20] (\x,2) to [bend right=20] (\x+2,2) to [bend right=20] (\x+2,0) to [bend right=20] (\x,0);
\draw[line width=0.7] (\x,0) to [bend right=20] (\x,2);
\draw[line width=0.7] (\x,2) to [bend right=20] (\x+2,2);
\draw[line width=0.7] (\x+2,2) to [bend right=20] (\x+2,0);
\draw[line width=0.7] (\x+2,0) to [bend right=20] (\x,0);
}

\tikzstyle{every node}=[inner sep=1.5pt, fill=black,circle,draw]

\node[red,ultra thick] (v2) at (2,0) {};
\node (v4) at (4,0) {};
\node (v6) at (6,0) {};
\node (v8) at (8,0) {};
\node (v10) at (10,0) {};
\node (v12) at (12,0) {};
\node (v14) at (14,0) {};
\node[red,ultra thick] (v16) at (16,0) {};
\node[red,ultra thick] (v18) at (18,0) {};
\node (v20) at (20,0) {};
\node (v22) at (22,0) {};
\node (v24) at (24,0) {};
\node (v26) at (26,0) {};
\node (v28) at (28,0) {};
\node (v30) at (30,0) {};
\node[red,ultra thick] (v32) at (32,0) {};

\node[red,ultra thick] (u2) at (2,2) {};
\node (u4) at (4,2) {};
\node (u6) at (6,2) {};
\node (u8) at (8,2) {};
\node (u10) at (10,2) {};
\node (u12) at (12,2) {};
\node (u14) at (14,2) {};
\node[red,ultra thick] (u16) at (16,2) {};
\node[red,ultra thick] (u18) at (18,2) {};
\node (u20) at (20,2) {};
\node (u22) at (22,2) {};
\node (u24) at (24,2) {};
\node (u26) at (26,2) {};
\node (u28) at (28,2) {};
\node (u30) at (30,2) {};
\node[red,ultra thick] (u32) at (32,2) {};

\node[red, ultra thick] (t3) at (3,4) {};
\node (t7) at (7,4) {};
\node (t11) at (11,4) {};
\node[red,ultra thick] (t15) at (15,4) {};
\node[red,ultra thick] (t19) at (19,4) {};
\node (t23) at (23,4) {};
\node (t27) at (27,4) {};
\node[red,ultra thick] (t31) at (31,4) {};

\draw[dotted,line width=1.4] (t3) to [bend left=20] (v6);
\draw[dotted,line width=1.4] (t3) to [bend left=20] (u6);
\draw[dotted,line width=1.4] (t7) to [bend left=20] (v10);
\draw[dotted,line width=1.4] (t7) to [bend left=20] (u10);
\draw[dotted,line width=1.4] (t11) to [bend left=20] (v14);
\draw[dotted,line width=1.4] (t11) to [bend left=20] (u14);
\draw[dotted,line width=1.4] (t15) to [bend left=20] (v18);
\draw[dotted,line width=1.4] (t15) to [bend left=20] (u18);
\draw[dotted,line width=1.4] (t19) to [bend left=20] (v22);
\draw[dotted,line width=1.4] (t19) to [bend left=20] (u22);
\draw[dotted,line width=1.4] (t23) to [bend left=20] (v26);
\draw[dotted,line width=1.4] (t23) to [bend left=20] (u26);
\draw[dotted,line width=1.4] (t27) to [bend left=20] (v30);
\draw[dotted,line width=1.4] (t27) to [bend left=20] (u30);

\draw[dotted,line width=1.4] (t7) to [bend right=20] (v4);
\draw[dotted,line width=1.4] (t7) to [bend right=20] (u4);
\draw[dotted,line width=1.4] (t11) to [bend right=20] (v8);
\draw[dotted,line width=1.4] (t11) to [bend right=20] (u8);
\draw[dotted,line width=1.4] (t15) to [bend right=20] (v12);
\draw[dotted,line width=1.4] (t15) to [bend right=20] (u12);
\draw[dotted,line width=1.4] (t19) to [bend right=20] (v16);
\draw[dotted,line width=1.4] (t19) to [bend right=20] (u16);
\draw[dotted,line width=1.4] (t23) to [bend right=20] (v20);
\draw[dotted,line width=1.4] (t23) to [bend right=20] (u20);
\draw[dotted,line width=1.4] (t27) to [bend right=20] (v24);
\draw[dotted,line width=1.4] (t27) to [bend right=20] (u24);
\draw[dotted,line width=1.4] (t31) to [bend right=20] (v28);
\draw[dotted,line width=1.4] (t31) to [bend right=20] (u28);

\node (s5) at (5,6) {};
\node (s13) at (13,6) {};
\node (s21) at (21,6) {};
\node (s29) at (29,6) {};

\draw (t3) -- (s5)--(t7);
\draw (t11) -- (s13)--(t15);
\draw (t19) -- (s21)--(t23);
\draw (t27) -- (s29)--(t31);

\node[red,ultra thick] (r9) at (9,8) {};
\node[red,ultra thick] (r25) at (25,8) {};

\draw (s5) -- (r9)--(s13);
\draw (s21) -- (r25)--(s29);

\node[red,ultra thick] (q17) at (17,10) {};
\draw (r9) -- (q17)--(r25);

\tikzstyle{every node}=[]
\draw[red] (17,9.2) node [] {$x=r$};
\draw[red,below] (9,7.7) node [] {$x_1$}; 
\draw[red,below] (25,7.7) node [] {$x_2$}; 
\draw[gray] (17,-0.6) node [] {\large $H_{i_x}^2$};

\end{tikzpicture}
    \vspace{-2em}
    \caption{Indicated in red is the bag $V_x$ from \cref{constr:TreeDecomp1} for the case $x=r$ and $h, m = 3$.}
    \label{fig:TreeDecomp}
\end{figure}

\begin{construction} \label{constr:TreeDecomp2}
    Let $h,d,m \in \N$ with $m \geq 2$. For $m = 2$, we let $(T, \cV)$ be the \td\ of $G := G_{h,d,2}$ given by \cref{constr:TreeDecomp1} (where $H_{i_x}^1$ is a path of length $d+1$ between two anchors).

    Now let $m > 2$, and assume that we have defined a \td\ of $G_{h,d,m-1}$. Let $G := G_{h,d,m}$, and let $(T', \cV')$ be the \td\ of $G$ provided by \cref{constr:TreeDecomp1}. For every $H^{m-1}_i$, let $(T^i, \cV^i)$ be the \td\ of $H^{m-1}_i$ provided by \cref{constr:TreeDecomp2} for $h,d,m-1$ (where we remark that~$H^{m-1}_i$ is isomorphic to $G_{h,d,m-1}$). 

    Let $T$ be obtained from the disjoint union of tree $T'$ and the trees $T^i$'s by identifying each leaf~$v_x$ of~$T'$ with the root~$r_{i_x}$ of~$T^{i_x}$. Moreover, we set $V_x := V'_x$ for every node $x \in V(T')$ that is not a leaf, and $V_x := V^i_x$ for every other node of $T$ where $i$ is the unique integer such that $x \in V(T^i)$. 
    Since the adhesion set of an edge $e = xv_x$ in $T'$ is $V_e = \SSS(H^{m-1}_{i_x}) \cup \TTT(H^{m-1}_{i_x}) \cup \{\RR(G)\}$ and this set is also contained in the bag $V^{i_x}_{r_{i_x}}$ associated with the root $r_{i_x}$ of $T^{i_x}$, it follows that $(T, \cV)$ is a \td\ of $G$. 
\end{construction}

We are now ready to prove \cref{lem:NoFatGrids}.

\begin{proof}[Proof of \cref{lem:NoFatGrids}]
    Let $G := G_{h,d,m}$. Suppose for a contradiction that $G$ contains a $3$-fat model $(\cU, \cE)$ of the $(154\times 154)$-grid.
    Let $W$ be a set of $154$ vertices, one from each branch set in $\cU$ that corresponds to a vertex of the $77$-th row in the $(154\times 154)$-grid.

    Let $(T, \cV)$ be the \td\ of $G$ provided by \cref{constr:TreeDecomp2}. 
    We distinguish two cases according to how $W$ interacts with the separations of $G$ induced by $(T, \cV)$. We remark that $|W| = 154$, and hence the two cases below cover all possibilities (though they are not mutually exclusive).
    \medskip

    \noindent \textbf{Case 1:} \textit{For every edge $e = xy$ of $T$ precisely one of the two sides of the separation $\{A^x_e, A^y_e\}$ induced by~$e$ contains at least $123$ vertices of $W$.}

    We orient each edge of $T$ towards the unique side of $\{A^x_e, A^y_e\}$ that contains at least $123$ vertices of $W$. This defines an orientation of $E(T)$. Since $T$ is finite, there is a sink $x \in V(T)$, i.e.\ a node~$x$ of~$T$ all whose incident edges are oriented towards $x$. Since $x$ has degree at most $3$ in $T$, at least $|W| - 3\cdot (|W|-123) = 61$ vertices of $W$ are contained in $V_x$. By \eqref{eq:Bags} and the pigeon hole principle, it follows that at least one of the four $V_i^j$'s contained in $V_x$ contains at least $14$ vertices of $W$. This contradicts \cref{lem:NoFatLinkedSet:Vi}.  
    \medskip

    \noindent \textbf{Case 2:} \textit{There is an edge $e = xy$ of $T$ such that both sides of the separation $\{A^x_e, A^y_e\}$ induced by~$e$ contain at least $32$ vertices of $W$.}

    Let $W_z$, for $z \in \{x,y\}$, be a set of $32$ vertices in $W \cap A^z_e$. Now recall that $W$ consists of one vertex from each branch set corresponding to a vertex of the $77$-th row of the $(154 \times 154)$-grid. Since the `upper half' of the $(154 \times 154)$-grid is isomorphic to the $(154 \times 77)$-grid, the $77$-th row of the $(154 \times 154)$-grid is $1$-fat $77$-path-connected in the `upper half' of the $(154 \times 154)$-grid. Since our model of the $(154 \times 154)$-grid is $3$-fat in~$G$, it follows that there are $32$ paths $P_1, \dots, P_{32}$ between $W_x$ and $W_y$ in $G$ that are pairwise at least $3$ apart in $G$ and such that each $P_i$ is contained in the union of branch sets and paths of $(\cU, \cE)$ corresponding to the `upper half' of the $(154 \times 154)$-grid. 

    Since $V_e$ separates $W_x \subseteq A^x_e$ from $W_y \subseteq A^y_e$ in $G$, each $P_i$ meets $V_e$. For every $i \in [32]$ pick a vertex $w_i \in V(P_i) \cap V_e$, and let $W'$ be the union of the $w_i$. We show that $W'$ is $3$-fat $16$-connected in $G$. For every $w_i \in W'$ let $p_i$ be one of the two endvertices of $P_i$. Then the $p_i$ lie in branch sets corresponding to distinct vertices of the $77$-th row of the $(154 \times 154)$-grid. 
    Now for every choice of two sets $X, X' \subseteq W'$ of size at most $16$ and with $|X|=|X'|$, applying \cref{obs:LinesInGridAreNConn} to the $77$-th row and the `lower half' of the $(154 \times 154)$-grid yields $16$ paths $Q_1, \dots, Q_{16}$ between $\{p_i : w_i \in X\}$ and $\{p_i : w_i \in X'\}$ that are pairwise at least $3$ apart in $G$ and that are contained in the union of the branch sets and paths of $(\cU, \cE)$ corresponding to the `lower half' of the $(154 \times 154)$-grid. This in particular means that the paths $Q_i$ are internally disjoint from the $P_i$, so we can extend the $Q_i$ to $X$--$X'$ paths by adding suitable subpaths of the~$P_i$. By construction, these extend paths are still pairwise at least $3$ apart in $G$. 

    Hence, $W'$ is $3$-fat $16$-path-connected in $G$; in particular, $|W'| = 32$. Since $W' \subseteq V_e$, it follows by~\eqref{eq:AdhesionSets} and the pigeon hole principle that there is some $V_i^j$ that meets at least $14$ vertices of $W'$. Then $W' \cap V_i^j$ is $3$-fat $7$-path-connected, which contradicts \cref{lem:NoFatLinkedSet:Vi}.
\end{proof}

For later use in the proof of \cref{main:ErdosPosaPlanar}, let us remark that one can prove along the same lines the following stronger version of \cref{lem:NoFatGrids}.

For every $h,d,m \in \N$, let $\tilde{G}_{h,d,m}$ be the graph obtained from $G := G_{h,d,m}$ by adding, for every $i \in [m]$, a path $W_i$ of length $d+1$ between $t_i \in \TTT(G)$ and $s_{m-i+1} \in \SSS(G)$ (and otherwise disjoint from $G$ and from all other $W_j$).

\begin{customlem}{\cref*{lem:NoFatGrids}$^\prime$} \label{lem:NoFatGrids:EP}
    The $(154 \times 154)$-grid is not a $3$-fat minor of $
    \tilde{G}_{h,d,m}$ for any $h,d,m \in \N$ with $m \geq 2$.
\end{customlem}

\begin{proof}[Proof sketch]
    The proof is essentially the same as for \cref{lem:NoFatGrids}: The proof of \cref{lem:NoFatLinkedSet:Vi} stays the same except that in the case where $V_i = V_1^m (=\SSS(G))$ we let $H_{i-1}$ be $H^{m-1}_n$ (that is the (unique) copy of $G_{h,d,m-1}$ with $\TTT(H^{m-1}_n) \subseteq \TTT(G)$), and in the case where $V_i = V_n^m (= \TTT(G))$ we let $H_{i}$ be $H_1^{m-1}$. 
    Moreover, since $\SSS(G) \cup \TTT(G) \subseteq V_r$ in the \td\ $(T, \cV)$ from \cref{constr:TreeDecomp2}, we can add for every path $W_i$ a node $u_i$ to $T$, make it incident to~$r$, and associate with $u_i$ the bag $V(W_i)$. This clearly yields a \td\ of $\tilde{G}_{h,d,m}$. Moreover, since no branch set of the $77$-th row of any model of the $(154 \times 154)$-grid can be contained inside a path, we can choose the set $W$ in the proof of \cref{lem:NoFatGrids} so that it is contained in $G$. Then in Case~1 in the proof of \cref{lem:NoFatGrids}, we still find a node $x$ whose bag~$V_x$ contains at least $61$ vertices of $W$ (as no newly added bag contains any vertices from $W$), and in Case~2, the edge $e$ will still be an edge of the original tree $T$. Therefore, the proof of \cref{lem:NoFatGrids} also yields \cref{lem:NoFatGrids:EP}.
\end{proof}

\section{No quasi-isometry to a \texorpdfstring{$K_n$}{Kn}-minor-free graph} \label{sec:KnMinors}

In this section we prove \cref{lem:KnMinors}, which we recall here for convenience:

\begin{customlem}{\cref*{lem:KnMinors}}
    For every $M, A, n \in \N$ with $M \geq 1$, there exist $h, d, m \in \N$ such that $G_{h,d,m}$ is not $(M,A)$-quasi-isometric to a graph with no $K_n$ minor.
\end{customlem}

We remark that as in \cite{CounterexAgelosPanosConjecture}, it is possible to strengthen \Cref{lem:KnMinors} so that $G_{h,d,m}$ is not $(M,A)$-quasi-isometric to a graph with no 2-fat $K_n$ minor.
However, we omit the details since the proof becomes much more technical, partly due to a trick we use of applying Menger's theorem as in \cite{CounterexAgelosPanosConjecture} not working in the 2-fat setting. 

Let us first give a brief sketch of the proof of \Cref{lem:KnMinors}. For this, let $M,A, n \in \N$, let $G:=G_{h,d,m}$ for some large enough $h,d,m \in \N$ that we will specify later in the proof, and let $f$ be an $(M,A)$-quasi-isometry from~$G$ to some graph~$H$. To conclude the proof, we need to find a model of $K_n$ in $H$.

We first show that in $H$, there are $m$ pairwise disjoint paths between $f(\SSS(G))$ and $f(\TTT(G))$ that avoid $B:=B_H(f(\RR(G)), q)$ for some carefully chosen $q \in \N$. For this, we employ Menger's theorem, which yields a set $U \subseteq V(H)$ of at most $m-1$ vertices that separates $f(\SSS(G))$ and $f(\TTT(G))$ in~$H-B$ otherwise. We can then use $f$ to translate $U$ back to a set of $m-1$ vertices in~$G$ such that bounded-radius balls around them separate $\SSS(G)$ from $\TTT(G)$ in~$G-B_G(\RR(G), q')$. But this contradicts that~$G$ is a counterexample to the weak coarse Menger conjecture (and, more precisely, it contradicts \cref{lem:NoSmallSeparator}).

Hence, there are $m$ pairwise disjoint paths $P_1, \dots, P_m$ between $f(\SSS(G))$ and $f(\TTT(G))$ in~$H-B$, where $m := n^2$. We further find $n(n-1)/2$ sets $S_i$ of size $m+1$ that are pairwise far apart and that each separate $\SSS(G)$ from $\TTT(G)$ in $G-B_G(\RR(G),q')$ (essentially, $S_i$ will consist of some $V_j^{m-1}$ together with two leaves of $B(G)$). It follows that every path $P_j$ meets a bounded-radius ball around every such~$f(S_i)$. In fact, we will have chosen the paths $P_i$ so that they intersect balls around distinct vertices in each~$S_i$. We then throw away those paths $P_j$ that come close to the image under~$f$ of a leaf of $B(G)$ in some~$S_i$. These are at most $n(n-1)$ many, so we end up with $n$ paths $P_j$ that each meet a bounded-radius ball around every $f(V_j^{m-1})$ contained in some $f(S_i)$, but which avoid a bounded-radius ball around the image under~$f$ of the leaf of~$B(G)$ which is connected to $V_j^{m-1}$ via spines.
We then make these $n$ paths $P_j$ the branch sets of our model of~$K_n$, and as branch paths we roughly choose images under~$f$ of two suitable spines.

\begin{proof}[Proof of \cref{lem:KnMinors}]
    Let $M, A, n \in \N$ so that $M \geq 1$. 
    Let $r = r(M,A)$ be given by \cref{lem:QIPreservesSeps}, and
    set 
    \begin{align*}
        N &:= \frac{n(n-1)}{2},\\
        q &:= \lceil \log_2(N+1)\rceil,\\
        d &:= 4rM((2r+q)M+(6r+1)A+1),\\
        h &:= d+2, \text{ and}\\
        m &:= n^2. 
    \end{align*}
    
    Set $G := G_{k,d,m}$. Let $f:V(G) \to V(H)$ be an $(M,A)$-quasi-isometry from $G$ to a graph~$H$.
    We denote $\SSS(H) := f(\SSS(G))$, $\TTT(H) := f(\TTT(G))$ and $\RR(H) := f(\RR(G))$. Note that by \ref{quasiisom:1} and \cref{obs:BlackDotsAreFarApart} and because $d \geq M(A+1)$, we have that $|\SSS(H)| = m = |\TTT(H)|$ and that $\SSS(H), \TTT(H)$ are disjoint and do not contain $\RR(H)$.

    By the choice of $q$, the $q$-th level of the binary tree $B(G)$ has at least $N+1$ vertices, which we consider to be enumerated from left to right as $x_0, \dots, x_N$. For every $i \in [N]$ (recall that $[N]$ does not contain $0$), let $\Delta_i := \Delta(x_i)$ and $S_i := \SSS(\Delta_i) \cup \{\ell_i\}$ where $\ell_i$ is the (unique) leaf of $B(G)$ that is joined to $\SSS(\Delta_i)$ via spines (equivalently, $\ell_i$ is the leaf of $B(G)$ left of the (unique) leaf of $B(G)$ contained in $\SSS(\Delta_i)$).\footnote{We excluded $0$ to ensure that $\ell_i$ exists.}
    Let us remark for later use that 
    \begin{enumerate}[label=\rm{(\roman*)}]
        \item \label{itm:Nabla:FarApart} $d_G(S_i, S_j) \geq M(A+2r+1)$ for all $i \neq j$, and
        \item \label{itm:Nabla:FarApart:2} $d_G(u,v) \geq M(A+2r+1)$ for all $u\neq v \in S_i$ and $i \in [N]$.
    \end{enumerate}
    Indeed, for every $u,v \in S_i \cup S_j$ we have $d_G(u, v) \geq 2d+2$ by \cref{obs:BlackDotsAreFarApart}, unless at least one of $u,v$, say~$v$, is a leaf of $B(G)$, in which case either $d_G(u, v) \geq d+1$ if a shortest $u$--$v$ path contains a spine, or $d_G(u, v) \geq 2(h-q)$ if a shortest $u$--$v$ path is contained in $B(G)$. Hence, the assertion follows by the choice of $d$ and $h$ (we remark that for this argument $d \geq M(A+2r+1)$ and $h \geq M(A+2r+1)/2+q$ would suffice, but we needed to choose $d,h$ much larger for other reasons, which will become apparent later).
    
    Let $U := \bigcup_{i \in [m]} S_i$.

    \begin{claim} \label{claim:ManySTPathsInH}
        There are $m$ pairwise disjoint $\SSS(H)$--$\TTT(H)$ paths $P_1, \dots, P_m$ in $H - B_H(\RR(H), Mq+A)$ such that, for every $u \in U$, at most one $P_i$ meets $B_H(f(u), r)$.
    \end{claim}

    \begin{claimproof}
        Set $q':=Mq+A$.
        Let $H'$ be obtained from $H$ by contracting $B_G(f(u), r)$ down to a vertex $h_u$ for every $u \in U$. Note that the contraction vertices $h_u$ are pairwise distinct by \ref{itm:Nabla:FarApart}, \ref{itm:Nabla:FarApart:2} and \ref{quasiisom:1} of $f$. Note further that $H$ is $(2r,0)$-quasi-isometric to $H'$, as witnessed by the map $g:V(H) \to V(H')$ that maps every $h \in V(H) \setminus B_H(f(U),r)$ to $h$ and every $h \in B_H(f(U), r)$ to the unique contraction vertex $h_u$ such that $h \in B_H(f(u), r)$.

        Then $f':=g \circ f$ is an $(M',A')$-quasi-isometry from $G$ to $H'$ where $M':= 2rM$ and $A':= 2rA$. In particular, we may define $\SSS(H'), \TTT(H')$ and $\RR(H')$ analogously for $H'$, e.g.\ $\SSS(H') := f'(\SSS(G))$. Set also $B := B_{H'}(\RR(H'), q')$.
        We claim that there are $m$ pairwise disjoint $\SSS(H')$--$\TTT(H')$ paths $P'_1, \dots, P'_m$ in $H'-B$.
        We then conclude the proof of the claim as follows. Since every $B_H(f(u), r)$ is connected by \cref{lem:QIPreservesConn} and because $r \geq M+A$, we obtain for every $i \leq m$ a path $P_i$ in $H$ by replacing each contraction vertex~$h_u$ in~$P'_i$ with a path through $B_H(f(u),r)$. Since the balls $B_H(f(u), r)$ are pairwise disjoint as shown above, the paths $P_i$ are pairwise disjoint. Moreover, every $B_H(f(u), r)$ meets at most one~$P_i$. Finally, observe that $d_G(u,\RR(G)) \geq h \geq M(q'+r+A+1)$ for all $u \in U$, which implies by \ref{quasiisom:1} that $d_H(f(u), \RR(H)) \geq q'+r+1$, and thus the $P_i$ lie in $H-B_H(\RR(H),q')$. Hence, the $P_i$ are as desired.
    
        Now suppose for a contradiction that we cannot find $m$ pairwise disjoint $\SSS(H')$--$\TTT(H')$ paths in $H' - B$. Then by Menger's theorem, there is a set $Y \subseteq V(H')\setminus B$ of size at most $m-1$ that intersects all $\SSS(H')$--$\TTT(H')$ paths in $H'-B$.     
        
        Since $f'$ is an $(M',A')$-quasi-isometry, there exists for every $y \in Y$ some $x \in V(G)$ such that $d_{H'}(y, f'(x)) \leq A'$, we fix such a vertex $x$ for every $y \in Y$, and let $X$ comprise all of them. 

        Set $\ell := M'(M'+3A'+q'+1)$. By \cref{lem:NoSmallSeparator} and because $h = 2\ell+2$ and $d = 2\ell$, there is an $\SSS(G)$--$\TTT(G)$ path $P$ in $G$ such that $d_G(P, X \cup \{\RR(G)\}) \geq \ell$. 
        Since $B_{H'}(f(P), M'+A')$ is connected by \cref{lem:QIPreservesConn}, there exists an $s$--$t$ path $Q$ in $H'[B_{H'}(f(P), M'+A')]$ where $s,t$ are the images under $f'$ of the endvertices of $P$ in $\SSS(G)$ and $\TTT(G)$, respectively. In particular, $s \in \SSS(H')$ and $t \in \TTT(H')$. 

        To obtain the desired contradiction, in remains to show that $Q$ avoids $Y \cup B$. We have
        \begin{align*}
        d_{H'}(Q, \RR(H')) &\geq d_{H'}(f'(P), f'(\RR(G))) - (M'+A')\\
        \intertext{since $Q \in B_{H'}(f'(P), M'+A')$, and hence, by \ref{quasiisom:1} and since $f'$ is an $(M',A')$-quasi-isometry,}
        d_{H'}(Q, \RR(H')) &\geq \Big(\frac{1}{M'}d_G(P, \RR(G))-A'\Big) -(M'+A') \geq \frac{1}{M'}(M'(M'+3A'+q+1))-M'-2A' > q'.
        \end{align*}
        Thus, $Q$ avoids $B = B_{H'}(f'(\RR(G)), q')$. Moreover, for every $y \in Y$ we have
        \begin{align*}
            d_{H'}(Q, y) &\geq d_{H'}(f'(P), f'(x)) - (M'+A')-A',\\
            \intertext{where $x \in X \subseteq V(G)$ is such that $d_{H'}(y, f'(x)) \leq A'$, and hence}
            d_{H'}(Q, y) &\geq \Big(\frac{1}{M'}d_G(P, x) - A'\Big) - M' -2 A' \geq \frac{1}{M'}(M'(M'+3A'+q+1))-M'-3A' \geq 1.
        \end{align*}
        Thus, $Q$ also avoids $Y$, which contradicts that $Y$ separates $\SSS(H')$ and $\TTT(H')$ in $H'-B$.
    \end{claimproof}

    Let $P_1, \dots, P_m$ be pairwise disjoint $\SSS(H)$--$\TTT(H)$ paths in $H - B_H(\RR(H), Mq+A)$ as provided by \cref{claim:ManySTPathsInH}; in particular, for every $u \in U$, at most one $P_i$ meets $B_H(f(u), r)$.

    Let $S'_i := B_H(f(S_i), r)$ for every $i \in [N]$.
    Then, by \ref{itm:Nabla:FarApart} and \ref{quasiisom:1} of $f$,
    \begin{enumerate}[label=\rm{(\arabic*)}]
        \setcounter{enumi}{1}
        \item \label{itm:Delta:FarApart} $d_H(S'_i, S'_j) \geq 1$ for all $i \neq j \in [N]$.
    \end{enumerate}

    For every $i \in [N]$, recall that $\ell_i$ is the (unique) leaf of $B(G)$ that is contained in $S_i$ but not in $\SSS(\Delta_i)$. Let $\ell'_i$ be the (unique) leaf of $S_i$ that is not $\ell_i$.
    By \cref{claim:ManySTPathsInH}, there are for every $i \in [N]$ at most two paths $P_j$ that meet $B_H(\{f(\ell_i), f(\ell'_i)\}, r)$. Hence, as $m = n^2 = 2N+n$ and we have $m$ paths~$P_j$, there are at least $n$ paths amongst the $P_j$ which avoid $\bigcup_{i \in [N]} B_H(\{f(\ell_i), f(\ell'_i)\},r)$. Without loss of generality assume that $P_1, \dots, P_n$ avoid $\bigcup_{i \in [N]} B_H(\{f(\ell_i), f(\ell'_i)\}, r)$.

    We now obtain a model of $K_n$ in $H$ as follows. Its branch sets $U_j$ are the sets $V(P_j)$ for $j \in [n]$. To define the branch paths, let $\{e_1, \dots, e_N\}$ be an enumeration of the edges in $K_n$. 
    Then for every edge $e_j = ik \in E(K_n)$ we choose a branch path $E_{e_j}$ as follows. Since $S'_j$ separates $\SSS(H)$ from $\TTT(H)$ in $H - B_H(\RR(H), Mq+A)$ by \cref{obs:SubtrianglesSeparateST} and \cref{lem:QIPreservesSeps}, both $P_i$ and $P_k$ meet $S'_j$. In particular, by the definition of $S'_j$ and because $P_i,P_k$ avoid $B_H(\{f(\ell_j),f(\ell'_i)\},r)$, both $P_i$ and $P_k$ meet $B_H(f(\SSS(\Delta_j),r)$. Let $a_i, a_k$ be vertices of $\SSS(\Delta_j)$ such that $P_i$ meets $B_H(f(a_i),r)$ and $P_k$ meets $B_H(f(a_k), r)$.
    Let $Q_{ji}, Q_{jk}$ be the spines of $G$ joining $\ell_j$ to $a_i, a_k$, respectively. By \cref{lem:QIPreservesConn}, $X_j := B_H(f(Q_{ji}) \cup f(Q_{jk}), r)$ is connected, and it is separated from $f(G-(Q_{ji} \cup Q_{jk}))$ by $B_H(\{f(\ell_j), f(a_i), f(a_k)\}, r)$ by \cref{lem:QIPreservesSeps}. Hence, no path~$P_a$ other than $P_i, P_k$ meets $X_i$ by \cref{claim:ManySTPathsInH}, so we can choose the branch path $E_{e_k}$ connecting $P_i$ and~$P_k$ inside~$X_j$. By the previous argument, $E_{e_j}$ avoids all other branch sets $V(P_a)$. Moreover, by \ref{itm:Nabla:FarApart}, the subsets~$X_j$ are disjoint, and hence also the branch paths $E_{e_j}$ are disjoint. Thus, we have found our desired model of $K_n$ in $H$. 
\end{proof}

\section{Counterexample to the coarse Erd\H{o}s-P\'{o}sa property for planar graphs}

In this section we prove \cref{main:ErdosPosaPlanar}, which we restate here for convenience:

\begin{customthm}{\cref*{main:ErdosPosaPlanar}} \label{main:ErdosPosaPlanar:copy}
    There exists a planar graph $X$ such that for every $K,D,n \in \N$ there is a graph $G'$ with the following properties:
    \begin{enumerate}[label=\rm{(\roman*)}]
        \item \label{itm:ErdosPosa:NoSmallHittingSet} For every set $Z$ of at most $n$ vertices of $G'$, there is a $K$-fat model of $X$ in $G'$ that avoids $B_{G'}(Z,D)$.
        \item \label{itm:ErdosPosa:NoThreeFatXs} $G'$ does not contain three $3$-fat models of $X$ that are pairwise at distance at least $3$ from each other.
    \end{enumerate}
\end{customthm}

We first describe the graph $X$ and the graphs $G'$, and then show that they satisfy the statement of \cref{main:ErdosPosaPlanar:copy}.

\begin{construction}[The graph $X$] \label{constr:X}
    \begin{figure}[ht]
        \centering
        \includegraphics[width=0.34\linewidth]{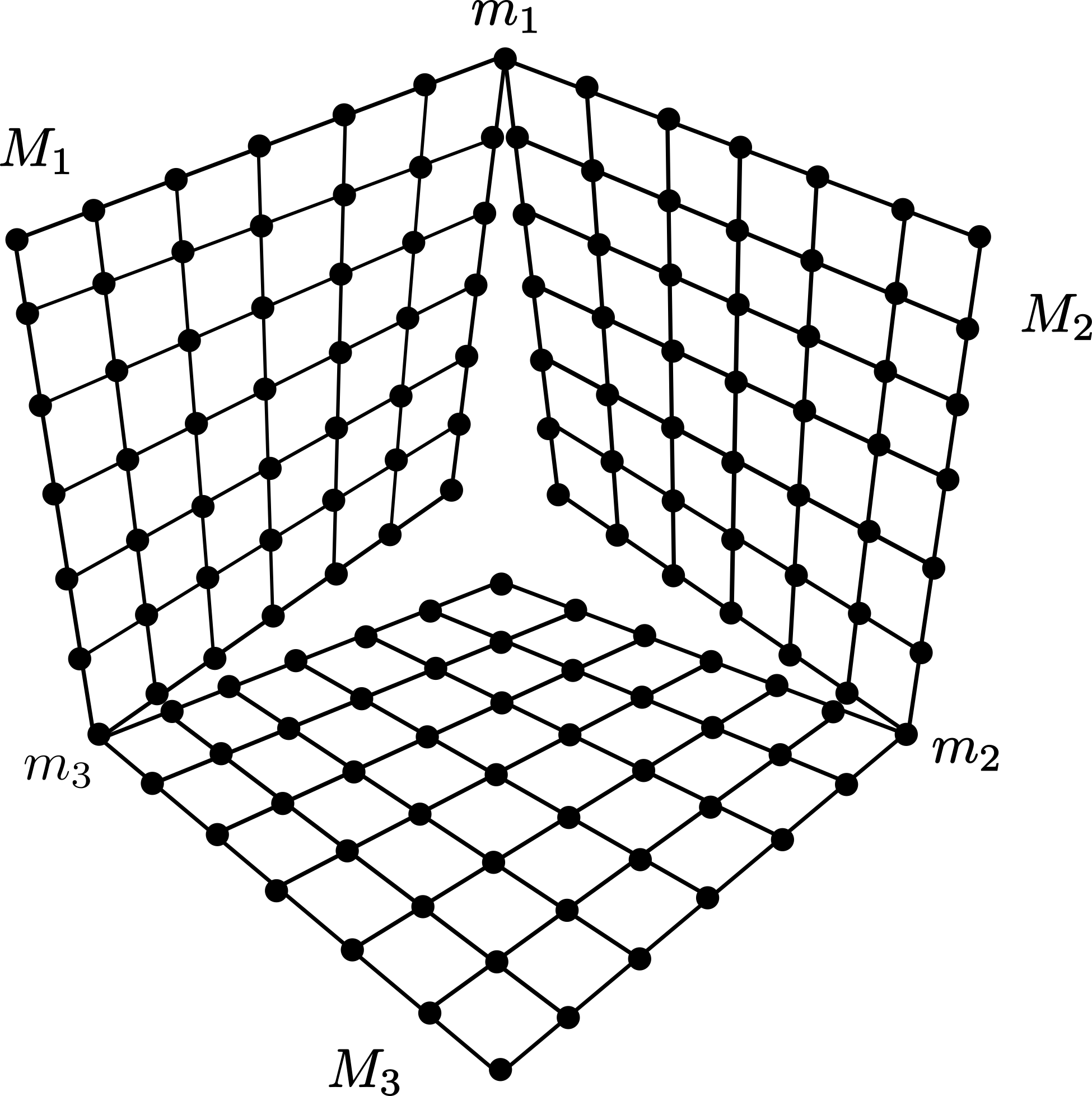}
        \caption{Depicted is the graph $X$ except that each of the three grids $M_1, M_2, M_3$ should be a $(154\times 154)$-grid (instead of a $(7 \times 7)$-grid).}
        
        \label{fig:ErdosPosa:X}
    \end{figure}
    Let $M_1, M_2, M_3$ be three disjoint copies of the $(154 \times 154)$-grid. Let \defn{$X$} be the graph obtained from $M_1\; \dot\cup\; M_2\; \dot\cup\; M_3$ by identifying, for every $i \in [3]$, the vertex of $M_i$ corresponding to $(154,154)$ with the vertex of $M_{i+1 \pmod{3}}$ corresponding to $(1,1)$ (see \cref{fig:ErdosPosa:X}). 

    We denote the identification vertices by $m_1,m_2,m_3$ so that $m_1$ is the (unique) vertex in $M_1 \cap M_2$, and $m_2$ is the (unique) vertex in $M_2 \cap M_3$, and $m_3$ is the (unique) vertex in $M_3 \cap M_1$ (see \cref{fig:ErdosPosa:X}).
\end{construction}

\begin{construction}[The graphs $G'$] \label{constr:Gprime}
    Let $h,d,m \in \N$. We define a graph \defn{$G'_{h,d,m}$} as follows. Let $L_1, \dots, L_m$ be disjoint copies of the subdivision of the $(154 \times 154)$-grid in which every edge is subdivided precisely $d$ times. Further, let $G$ be a copy of $G_{h,d,m}$. We (re-)enumerate the vertices in $\SSS(G)$ from `top' to `bottom', while we keep the enumeration of the the vertices in $\TTT(G)$ as it is (that means from `bottom' to `top') (see \cref{fig:ErdosPosa:G}). 
    Let $G'_{h,d,m}$ be the graph obtained from the disjoint union of $G$ and the $L_i$'s, for $i \in [m]$, by identifying, for every $i \in [m]$, the vertex of $L_i$ which correspond to the vertex $(1,1)$ of the $(154 \times 154)$-grid with $s_{i} \in \SSS(G)$, and the the vertex of $L_i$ that corresponds to $(154, 154)$ with $t_{i} \in \TTT(G)$ (see \cref{fig:ErdosPosa:G}).

    We use terms defined for $G = G_{h,d,m}$ also for $G'_{h,d,m}$, e.g.\ $\RR(G'_{h,d,m})$ is the root $\RR(G)$ of $G \subseteq G'_{h,d,m}$. 
\end{construction}

    \vspace{-5em}
    \begin{figure}[ht]
        \centering
        \begin{tikzpicture}[scale=0.49,auto=left]

\definecolor{dgreen}{rgb}{0,0.7,0}
\definecolor{dgrey}{rgb}{0.5,0.5,0.5}
\definecolor{lgrey}{rgb}{0.6,0.6,0.6}
\definecolor{llgrey}{rgb}{0.7,0.7,0.7}

\foreach \x in {2,4,6,8,10,12,14} {
\path [fill=llgrey,opacity=0.5]
(\x,0) to [bend right=20] (\x,1) to [bend right=20] (\x,2) to [bend right=20] (\x+2,2) to [bend right=20] (\x+2,1) to [bend right=20] (\x+2,0) to [bend right=20] (\x,0);
\draw[dgrey,line width=0.7] (\x,0) to [bend right=20] (\x,1);
\draw[dgrey,line width=0.7] (\x,1) to [bend right=20] (\x,2);
\draw[dgrey,line width=0.7] (\x,2) to [bend right=20] (\x+2,2);
\draw[dgrey,line width=0.7] (\x+2,2) to [bend right=20] (\x+2,1);
\draw[dgrey,line width=0.7] (\x+2,1) to [bend right=20] (\x+2,0);
\draw[dgrey,line width=0.7] (\x+2,0) to [bend right=20] (\x,0);
}

\tikzstyle{every node}=[gray, inner sep=1.5pt, fill=gray,circle,draw]

\node[orange,ultra thick] (v2) at (2,0) {};
\node (v4) at (4,0) {};
\node (v6) at (6,0) {};
\node (v8) at (8,0) {};
\node (v10) at (10,0) {};
\node (v12) at (12,0) {};
\node (v14) at (14,0) {};
\node[blue, ultra thick] (v16) at (16,0) {};

\node[ForestGreen,ultra thick] (w2) at (2,1) {};
\node (w4) at (4,1) {};
\node (w6) at (6,1) {};
\node (w8) at (8,1) {};
\node (w10) at (10,1) {};
\node (w12) at (12,1) {};
\node (w14) at (14,1) {};
\node[magenta, ultra thick] (w16) at (16,1) {};

\node[magenta,ultra thick] (u2) at (2,2) {};
\node (u4) at (4,2) {};
\node (u6) at (6,2) {};
\node (u8) at (8,2) {};
\node (u10) at (10,2) {};
\node (u12) at (12,2) {};
\node (u14) at (14,2) {};
\node[ForestGreen, ultra thick] (u16) at (16,2) {};

\node[blue, ultra thick] (t3) at (3,4) {};
\node (t7) at (7,4) {};
\node (t11) at (11,4) {};
\node[orange, ultra thick] (t15) at (15,4) {};

\draw[dgrey,dotted,line width=1.2] (t3) to [bend left=20] (v6);
\draw[dgrey,dotted,line width=1.2] (t3) to [bend left=20] (u6);
\draw[dgrey,dotted,line width=1.2] (t3) to [bend left=20] (w6);
\draw[dgrey,dotted,line width=1.2] (t7) to [bend left=20] (v10);
\draw[dgrey,dotted,line width=1.2] (t7) to [bend left=20] (w10);
\draw[dgrey,dotted,line width=1.2] (t7) to [bend left=20] (u10);
\draw[dgrey,dotted,line width=1.2] (t11) to [bend left=20] (v14);
\draw[dgrey,dotted,line width=1.2] (t11) to [bend left=20] (w14);
\draw[dgrey,dotted,line width=1.2] (t11) to [bend left=20] (u14);

\draw[dgrey,dotted,line width=1.2] (t7) to [bend right=20] (v4);
\draw[dgrey,dotted,line width=1.2] (t7) to [bend right=20] (w4);
\draw[dgrey,dotted,line width=1.2] (t7) to [bend right=20] (u4);
\draw[dgrey,dotted,line width=1.2] (t11) to [bend right=20] (v8);
\draw[dgrey,dotted,line width=1.2] (t11) to [bend right=20] (w8);
\draw[dgrey,dotted,line width=1.2] (t11) to [bend right=20] (u8);
\draw[dgrey,dotted,line width=1.2] (t15) to [bend right=20] (v12);
\draw[dgrey,dotted,line width=1.2] (t15) to [bend right=20] (w12);
\draw[dgrey,dotted,line width=1.2] (t15) to [bend right=20] (u12);

\node (s5) at (5,6) {};
\node (s13) at (13,6) {};

\draw[dgrey] (t3) -- (s5)--(t7);
\draw[dgrey] (t11) -- (s13)--(t15);

\node (r9) at (9,8) {};

\draw[dgrey] (s5) -- (r9)--(s13);

\draw[very thick] (4.2,6) to [bend right=28] (1.5,-0.5);
\draw [very thick] plot [smooth, tension=0.4] coordinates {(4.2,6) (9,8.5) (13.8,6)};
\draw[very thick] (13.8,6) to [bend left=28] (16.5,-0.5);
\draw[very thick] (1.5,-0.5) -- (16.5,-0.5);

\draw[dotted,blue,very thick] (3,4) to [bend left=80] (14,8);
\draw[dotted,blue,very thick] (14,8) to [bend left=60] (16,0);

\draw[dotted,blue,very thick] (5,7.9) to [bend left=50] (18.5,9);
\draw[dotted,blue,very thick] (9,9.6) to [bend left=40] (18.7,7);
\draw[dotted,blue,very thick] (12.4,9.4) to [bend left=40] (18.3,4);

\draw[dotted,blue,very thick] (4.7,9.2) to [bend left=60] (15.8,6.8);
\draw[dotted,blue,very thick] (9,12) to [bend left=50] (17,5);
\draw[dotted,blue,very thick] (13,13) to [bend left=45] (17.3,3.4);

\draw[dotted,blue,very thick] (3,4) to [bend left=80] (18,11);
\draw[dotted,blue,very thick] (18,11) to [bend left=30] (16,0);

\node[blue] (g1) at (14,8) {};
\node[blue] (g2) at (18,11) {};
\node[blue] (g3) at (5,7.9) {};
\node[blue] (g4) at (9.1,9.6) {};
\node[blue] (g5) at (12.4,9.4) {};

\node[blue] (g6) at (4.7,9.2) {};
\node[blue] (g7) at (9,12.05) {};
\node[blue] (g8) at (13,13) {};
\node[blue] (g9) at (18.55,9) {};
\node[blue] (g10) at (18.75,7) {};

\node[blue] (g6) at (18.3,4) {};
\node[blue] (g7) at (15.8,6.8) {};
\node[blue] (g8) at (16.9,5) {};
\node[blue] (g9) at (17.2,3.4) {};

\node[blue] (g6) at (9.15,10.95) {};
\node[blue] (g7) at (12.7,11.5) {};
\node[blue] (g8) at (16.25,10.65) {};

\node[blue] (g6) at (12.3,10.4) {};
\node[blue] (g7) at (14.95,10) {};
\node[blue] (g8) at (17.35,8.55) {};

\node[blue] (g6) at (14.4,9) {};
\node[blue] (g7) at (16.5,7.75) {};
\node[blue] (g8) at (17.8,5.85) {};


\draw[dotted,orange,very thick] (15,4) to [bend right=15] (15.5,5.5);
\draw[dotted,orange,very thick] (15,4) to [bend right=15] (16,5);

\draw[dotted,ForestGreen,very thick] (16,2) to [bend right=15] (16.4,3.5);
\draw[dotted,ForestGreen,very thick] (16,2) to [bend right=15] (16.9,3.3);

\draw[dotted,magenta,very thick] (16,1) to [bend right=15] (16.7,2.2);
\draw[dotted,magenta,very thick] (16,1) to [bend right=15] (16.95,2.1);

\draw[dotted,magenta,very thick] (2,2) to [bend left=15] (1.6,3.5);
\draw[dotted,magenta,very thick] (2,2) to [bend left=15] (1.1,3.3);

\draw[dotted,ForestGreen,very thick] (2,1) to [bend left=15] (1.2,2.3);
\draw[dotted,ForestGreen,very thick] (2,1) to [bend left=15] (0.9,2.05);

\draw[dotted,orange,very thick] (2,0) to [bend left=15] (1.1,1.2);
\draw[dotted,orange,very thick] (2,0) to [bend left=15] (0.8,0.9);

\tikzstyle{every node}=[]
\draw[gray] (9,6.9) node [] {$r=\RR(G)$};
\draw[orange,below right] (2,0.3) node [] {$s_4$};
\draw[magenta,above right] (2,1.85) node []  {$s_2$};
\draw[ForestGreen, right] (w2) node [] {$s_3$};
\draw[blue,below] (3,3.8) node [] {$s_1$};

\draw[blue,below left] (16,0.3) node [] {$t_1$};
\draw[ForestGreen,above left] (16,1.85) node []  {$t_3$};
\draw[magenta, left] (w16) node [] {$t_2$};
\draw[orange,below] (15,3.8) node [] {$t_4$};

\draw[blue] (19.5,10) node [] {$L_1$};
\draw (11.6,8) node [] {$G$};

\end{tikzpicture}
        \caption{Depicted is the graph $G_{2, d, 4}$ except that there should also be a subdivided grid~$L_i$ between $s_i$ and $t_i$ for every $i \in \{2,3,4$\}, and every $L_i$ should be a subdivided $(154 \times 154)$-grid (instead of a $(5\times 5)$-grid). Every dotted path has length $d+1$.}
        \label{fig:ErdosPosa:G}
    \end{figure}

We prove \cref{main:ErdosPosaPlanar} with the graph $X$ from \cref{constr:X} and with the graphs $G'$ defined in \cref{constr:Gprime}. We divide the proof into two lemmas.

\begin{lemma} \label{lem:ErdosPosa:NoSmallHittingSet}
    Let $X$ be the graph from \cref{constr:X}. For every $K,D,n \in \N$ there exist some $h,d,m \in \N$ such that for every set $Z$ of at most $n$ vertices of $G' := G'_{h,d,m}$, there is a $K$-fat model of $X$ in $G'$ that avoids $B_{G'}(Z,D)$. 
\end{lemma}

\begin{proof}
    Let $K,D,n \in \N$ be given, and let $h = h(K,D) := 2\cdot\max\{K,D\}+2$, $d = d(K,D) := \max\{3K,2D\}$ and $m = m(n):= n+3$. Set $G' := G'_{h,d,m}$, and let $G$ be the (unique) subgraph of $G'$ that is isomorphic to~$G_{h,d,m}$.

    Since $m = n+3$ and any two vertices in $\SSS(G') \cup \TTT(G')$ have distance at least $3d+3 > 2d+2D+2$, there are at least three indices $i_1, i_2, i_3 \in [m]$ such that none of $s_{i_j}, t_{i_j}$, for $j \in [3]$, lies in $B_{G'}(Z, d+D+1)$, and moreover, none of $L_{i_1}, L_{i_2}, L_{i_3}$ meets $B_{G'}(Z,D)$. We choose $i_1, i_2, i_3$ so that $i_1 < i_2 < i_3$ and so that they are as small as possible.

    For every $j \in [3]$ there is a natural way to obtain a $K$-fat model of $M_j$ inside $L_{i_j}$: for every vertex $x \in V(M_j)$, we let $U^j_x := B_{G'}(x', K)$ where $x'$ is the vertex of $L_{i_j}$ `corresponding' to $x$. (For this recall that $M_j$ is a copy and $L_{i_j}$ is a subdivision of the $(154 \times 154)$-grid.) Since $d \geq 3K$, the sets~$U_x$ have pairwise distance at least $K$ from each other. Moreover, every edge $xy \in E(M_j)$ `corresponds' to a path of length $d \geq 3K$ between $x'$ and $x'$ in $L_{i_j}$; we let $E^j_{xy}$ be the middle segment of this path of length $d-2K-2$ (which then starts in $U_x$ and ends in $U_y$).
    Finally, we rename the branch sets and paths of the model $(\cU^2, \cE^2)$ of~$M_2$ so that $s_{i_2} \in U^2_{m_{2}}$ (instead of $U^2_{m_1}$) and $t_{i_2} \in U^2_{m_1}$ (instead of $U^2_{m_2}$).
    It is straightforward to check that $(\cU^1\cup\cU^2\cup\cU^3, \cE^1\cup\cE^2\cup\cE^3)$ is a $K$-fat model of the disjoint union $M_1\, \dot\cup\, M_2\, \dot\cup\, M_3$. 
    
    To conclude the proof, it now suffices to make each of the sets $U^1_{m_1} \cup U^2_{m_1}$, $U^2_{m_2} \cup U^3_{m_2}$ and $U^3_{m_3} \cup U^1_{m_3}$ connected (while keeping them at distance at least $K$). For this, we have to find three paths $P_1,P_2,P_3$ that have pairwise distance at least~$K$ from each other so that $P_1$ starts in $t_{i_1} \in U^1_{m_1}$ and ends in $t_{i_2} \in U^2_{m_1}$, and $P_2$ starts in $s_{i_2} \in U^2_{m_2}$ and ends in $s_{i_3} \in U^3_{m_2}$, and $P_3$ starts in $t_{i_3} \in U^3_{m_3}$ and ends in $s_{i_1} \in U^1_{m_3}$ (see \cref{fig:ErdosPosa:Paths}).

    \begin{figure}[ht]
        \centering
        \begin{tikzpicture}[scale=0.49,auto=left]

\definecolor{dgreen}{rgb}{0,0.7,0}
\definecolor{dgrey}{rgb}{0.5,0.5,0.5}
\definecolor{lgrey}{rgb}{0.6,0.6,0.6}
\definecolor{llgrey}{rgb}{0.7,0.7,0.7}

\foreach \x in {8,12,18} {
\draw[dgrey,line width=0.7] (\x,0) to [bend right=20] (\x,1.5);
\draw[dgrey,line width=0.7] (\x,1.5) to [bend right=20] (\x,3);
\draw[dgrey,line width=0.7] (\x,3) to [bend right=20] (\x,4.5);
\draw[dgrey,line width=0.7] (\x,4.5) to [bend right=5] (\x+4,4.5);
\draw[dgrey,line width=0.7] (\x+4,4.5) to [bend right=20] (\x+4,3);
\draw[dgrey,line width=0.7] (\x+4,3) to [bend right=20] (\x+4,1.5);
\draw[dgrey,line width=0.7] (\x+4,1.5) to [bend right=20] (\x+4,0);
\draw[dgrey,line width=0.7] (\x+4,0) to [bend right=5] (\x,0);
}

\draw[dgrey,line width=0.7] (2,0) to [bend right=20] (2,1.5);
\draw[dgrey,line width=0.7] (2,1.5) to [bend right=20] (2,3);
\draw[dgrey,line width=0.7] (2,3) to [bend right=20] (2.2,4.5);
\draw[dgrey,line width=0.7] (2.2,4.5) to [bend right=5] (8,4.5);
\draw[dgrey,line width=0.7] (8,4.5) to [bend right=20] (8,3);
\draw[dgrey,line width=0.7] (8,3) to [bend right=20] (8,1.5);
\draw[dgrey,line width=0.7] (8,1.5) to [bend right=20] (8,0);
\draw[dgrey,line width=0.7] (8,0) to [bend right=0] (2,0);

\draw[dgrey,line width=0.7] (32,0) to [bend left=20] (32,1.5);
\draw[dgrey,line width=0.7] (32,1.5) to [bend left=20] (32,3);
\draw[dgrey,line width=0.7] (32,3) to [bend left=20] (31.8,4.5);
\draw[dgrey,line width=0.7] (31.8,4.5) to [bend left=20] (31.4,6);
\draw[dgrey,line width=0.7] (31.4,6) to [bend left=5] (26,6);
\draw[dgrey,line width=0.7] (26,6) to [bend left=8] (26,4.5);
\draw[dgrey,line width=0.7] (26,4.5) to [bend left=20] (26,3);
\draw[dgrey,line width=0.7] (26,3) to [bend left=20] (26,1.5);
\draw[dgrey,line width=0.7] (26,1.5) to [bend left=20] (26,0);
\draw[dgrey,line width=0.7] (26,0) to [bend left=0] (32,0);

\draw[dgrey,line width=0.7] (22,0) to [bend right=20] (22,1.5);
\draw[dgrey,line width=0.7] (22,1.5) to [bend right=20] (22,3);
\draw[dgrey,line width=0.7] (22,3) to [bend right=20] (22,4.5);
\draw[dgrey,line width=0.7] (22,4.5) to [bend right=20] (22,6);
\draw[dgrey,line width=0.7] (22,6) to [bend right=8] (26,6);
\draw[dgrey,line width=0.7] (26,6) to [bend right=20] (26,4.5);
\draw[dgrey,line width=0.7] (26,4.5) to [bend right=20] (26,3);
\draw[dgrey,line width=0.7] (26,3) to [bend right=20] (26,1.5);
\draw[dgrey,line width=0.7] (26,1.5) to [bend right=20] (26,0);
\draw[dgrey,line width=0.7] (26,0) to [bend right=20] (26,0);
\draw[dgrey, line width=0.7] (26,0) to [bend right=8] (22,0);

\foreach \x in {2,4} {
\draw[dgrey,line width=0.7] (\x,0) to [bend right=20] (\x,1.5) to [bend right=8] (\x+2,1.5) to [bend right=20] (\x+2,0) to [bend right=8] (\x,0);
}

\foreach \x in {28,30} {
\draw[dgrey,line width=0.7] (\x,0) to [bend right=20] (\x,1.5) to [bend right=20] (\x,3) to [bend right=8] (\x+2,3) to [bend right =20] (\x+2,1.5) to [bend right=20] (\x+2,0) to [bend right=8] (\x,0);
}

\tikzstyle{every node}=[gray, inner sep=1.5pt, fill=gray,circle,draw]

\node[ForestGreen,ultra thick] (v2) at (2,0) {};
\node[black] (w2) at (2,1.5) {};
\node[magenta,ultra thick] (u2) at (2,3) {};
\node[black] (x2) at (2.2,4.5) {};
\node[blue, ultra thick] (y2) at (2.6,6) {};
\node[black] (z2) at (3.2,7.5) {};

\node[black] (v3) at (32,0) {};
\node[blue, ultra thick] (w3) at (32,1.5) {};
\node[black] (u3) at (32,3) {};
\node[magenta, ultra thick] (x3) at (31.8,4.5) {};
\node[black] (y3) at (31.4,6) {};
\node[ForestGreen, ultra thick] (z3) at (30.8,7.5) {};

\node () at (8,0) {};
\node () at (8,1.5) {};
\node () at (8,3) {};
\node () at (8,4.5) {};

\node () at (12,0) {};
\node () at (12,1.5) {};
\node[ForestGreen, ultra thick] () at (12,3) {};
\node () at (12,4.5) {};

\node () at (16,0) {};
\node[ForestGreen,ultra thick] () at (16,1.5) {};
\node () at (16,3) {};
\node () at (16,4.5) {};

\node () at (18,0) {};
\node () at (18,1.5) {};
\node[ForestGreen,ultra thick] () at (18,3) {};
\node () at (18,4.5) {};

\node () at (22,0) {};
\node[ForestGreen, ultra thick] () at (22,1.5) {};
\node () at (22,3) {};
\node () at (22,4.5) {};
\node () at (22,6) {};

\node () at (26,0) {};
\node () at (26,1.5) {};
\node () at (26,3) {};
\node () at (26,4.5) {};
\node () at (26,6) {};

\node () at (4,0) {};
\node[magenta,ultra thick] () at (4,1.5) {};
\node () at (6,0) {};
\node[magenta,ultra thick] () at (6,1.5) {};

\node[blue,ultra thick] () at (30,0) {};
\node () at (30,1.5) {};
\node () at (30,3) {};
\node () at (28,0) {};
\node[blue,ultra thick] () at (28,1.5) {};
\node () at (28,3) {};

\draw[ForestGreen,very thick] (y2) to [bend left=20] (12,3);
\draw[ForestGreen,very thick] (z3) to [bend right=38] (22,1.5);
\draw[ForestGreen,very thick] (12,3) to [bend right=5] (16,1.5);
\draw[ForestGreen,very thick,dotted] (16,1.5) to [bend right=10] (18,3);
\draw[ForestGreen,very thick] (18,3) to [bend left=7] (22,1.5);

\draw[magenta,very thick] (2,3) to [bend left=20] (6,1.5) to [bend left=30] (4,1.5) to [bend left=20] (2,0);
\node[ForestGreen,ultra thick] () at (2,0) {};

\draw[blue,very thick] (x3) to [bend right=20] (28,1.5) to [bend right=20] (30,0) to [bend left=15] (32,1.5);

\draw[very thick] (3,8) to [bend right=20] (1.5,-0.5);
\draw [very thick] plot [smooth, tension=0.4] coordinates {(3,8) (17,11) (31,8)};
\draw[very thick] (31,8) to [bend left=20] (32.5,-0.5);
\draw[very thick] (1.5,-0.5) -- (32.5,-0.5);

\tikzstyle{every node}=[]
\draw[blue,below right] (y2) node [] {$s_{i_1}$};
\draw[magenta,above right] (2.1,3) node []  {$s_{i_2}$};
\draw[ForestGreen,above right] (2,0.2) node [] {$s_{i_3}$};

\draw[ForestGreen,below] (z3) node [] {$t_{i_3}$};
\draw[magenta,above left] (x3) node []  {$t_{i_2}$};
\draw[blue,above left] (32,1.5) node [] {$t_{i_1}$};

\draw[ForestGreen] (8,6) node [] {\large $P_3$};
\draw[magenta] (5,3.2) node [] {\large $P_2$};
\draw[blue] (29,4) node [] {\large $P_1$};
\draw[dgrey] (17.1,3.5) node [] {\LARGE $\dots$};
\draw[dgrey] (17.1,1) node [] {\LARGE $\dots$};
\draw[dgrey] (7,0.75) node [] {\LARGE $\dots$};
\draw[dgrey] (27.1,0.75) node [] {\LARGE $\dots$};
\draw[dgrey] (27.1,2.25) node [] {\LARGE $\dots$};

\draw[ForestGreen,below left] (12,3) node [] {$v_3$};
\draw[ForestGreen,below left] (16,1.5) node [] {$v_4$};
\draw[ForestGreen,above right] (18,3) node [] {$v_{c'}$};
\draw[ForestGreen,below left] (22,1.5) node [] {$v_{c'+1}$};

\draw[dgrey] (3,-1.2) node [] {$H^{j_2-1}_1$};
\draw[dgrey] (5,-1.2) node [] {$H^{j_2-1}_2$};
\draw[dgrey] (29,-1.2) node [] {$H^{i_2-1}_{c''-1}$};
\draw[dgrey] (31,-1.2) node [] {$H^{i_2-1}_{c''}$};

\draw[dgrey] (5,5) node [] {$H^{j_1-1}_1$};
\draw[dgrey] (11,5) node [] {$H^{j_1-1}_2$};
\draw[dgrey] (14,5) node [] {$H^{j_1-1}_3$};
\draw[dgrey] (20,5) node [] {$H^{j_1-1}_{c'}$};
\draw[dgrey] (24,6.5) node [] {$H^{i_1-1}_{c-1}$};
\draw[dgrey] (29,6.5) node [] {$H^{i_1-1}_c$};

\draw (25,10.5) node [] {\large $G$};

\end{tikzpicture}
        \caption{Depicted are the paths $P_1, P_2, P_3$ in $G$ from the proof of \cref{lem:ErdosPosa:NoSmallHittingSet}.}
        \label{fig:ErdosPosa:Paths}
    \end{figure}

    We first choose the path $P_3$ between $s_{i_1}$ and $t_{i_3}$. 
    Set $j_1 := m-i_1+1$ (so $j_1$ is the index of $s_{i_1}$ in the enumeration of $\SSS(G_{h,d,m})$ from `bottom' to `top' (instead of `top' to `bottom')).
    Since the situation is symmetric in $s_{i_1}$ and $t_{i_3}$, we may assume without loss of generality that $j_1 \leq i_3$. Let $c$  
    be the number of subgraphs $H_k^{i_3-1}$ of $G$ that are isomorphic to $G_{h, d, i_3-1}$, and let $c'$ 
    be the number of subgraphs $H_k^{j_1-1}$ of $H_1^{i_3-1} \cup \dots \cup H_{c-2}^{i_3-1}$ that are isomorphic to of $G_{h,d,j_1-1}$ (see \cref{fig:ErdosPosa:Paths}). 

    Since we chose $i_1$ as small as possible, there is for every $i < i_1$ some $z \in Z$ such that $d_{G'}(z, \{s_{i_1}, t_{i_1}\}) \leq d+D+1$. This implies that at most $n-i_1+1 = j_1-3$ vertices of $Z$ can have distance less than $d+D+1$ to $H_3^{j_1-1} \cup \dots \cup H_{c'}^{j_1-1}$. Hence, for every $k \in \{3, \dots, c'+1\}$ there exists some $v_k \in V_k^{j_1-1} = \SSS(H^{j_1-1}_k)$ that avoids $B_{G'}(Z, d+D+1)$.
    By \cref{lem:NoSmallSeparator} (more precisely, statement \eqref{eq:NoSmallSep:IndHyp} in the proof of \cref{lem:NoSmallSeparator}) there is for every $k \in \{3, \dots, c'\}$ a path in $H^{j_i-1}_k$ between $v_k$ and $v_{k+1}$ that avoids $B_{G'}(Z \cup \{\RR(G')\},D)$. Concatenating these paths yields a path $P'_3$ from $V^{j_1-1}_k$ to $V^{j_1-1}_{c'+1}$. Since $s_{i_1}$ and $t_{i_3}$ avoid $B_{G'}(Z, d+D+1)$, all spines from $s_{i_1}$ to $V^{j_1-1}_{3}$ and from $t_{i_3}$ to $V^{j_1-1}_{c'+1}$ avoid $B_{G'}(Z, D)$. Hence, we may extend $P'_3$ to an $s_{i_1}$--$t_{i_3}$ path which is contained in $H^{j_1-1}_3 \cup \dots \cup H^{j_1-1}_{c'}$ together with the spines starting in $s_{i_1}$ or $t_{i_3}$ (see \cref{fig:ErdosPosa:Paths}). In particular, $P_3$ avoids $B_{G'}(Z \cup \{\RR(G')\}, D)$ and $H^{j_1-1}_1 \cup H^{i_3-1}_c$.

    Next, we choose the path $P_2$ between $s_{i_2}$ and $s_{i_3}$. For this, let $j_2 := m-i_2+1$, and let $H^{j_2-1}_1, H^{j_2-1}_2$ be the first two (left-most) subgraphs of $G$ that are isomorphic to $G_{h,d,j_2-1}$ (see \cref{fig:ErdosPosa:Paths}). Similar as above, there exist two vertices $v_2 \in V^{j_2-1}_2$ and $v_3 \in V^{j_2-1}_3$ that avoid $B_{G'}(Z, d+D+1)$. Then again by (the statement \eqref{eq:NoSmallSep:IndHyp} in the proof of) \cref{lem:NoSmallSeparator}, there exist a path in $H^{j_2-1}_1$ between $s_{i_3}$ and $v_2$ and a path in $H^{j_2-1}_2$ between $v_2$ and $v_3$ that avoid $B_{G'}(Z \cup \{\RR(G')\}, D)$. By concatenating these two paths and adding the spine from $s_{i_2}$ to $v_3$ (which avoids $B_{G'}(Z,D)$ by the choice of $s_{i_2}$), we obtain an $s_{i_2}$--$s_{i_3}$ path $P_2$ that avoids $B_{G'}(Z \cup \{\RR(G')\}, D)$, and that is contained in $H^{j_2-1}_1 \cup H^{j_2-1}_2$ (see \cref{fig:ErdosPosa:Paths}).

    Finally, we choose $P_1$ analogously to $P_2$, except that $P_1$ has endvertices $t_{i_2}$ and $t_{i_1}$, and is contained in $H^{i_2-1}_{c''-1} \cup H^{i_2-1}_{c''}$ where $c''$ is the number of subgraphs $H^{i_2-1}_k$ of $G$ that are isomorphic to $G_{h,d,i_2-1}$ (and hence $H^{i_2-1}_{c''-1}, H^{i_2-1}_{c''}$ are the last two (right-most) such subgraphs of $G$; see \cref{fig:ErdosPosa:Paths}). In particular, $P_1$ is contained in $H^{i_2-1}_{c''-1} \cup H^{i_2-1}_{c''}$.

    Since $P_2 \subseteq H^{j_2-1}_1 \cup H^{j_2-1}_2$, and $P_1 \subseteq H^{i_2-1}_{c''-1} \cup H^{i_2-1}_{c''}$, and $P_1$ avoids $H^{j_1-1}_1 \cup H^{i_3-1}_c$ (see \cref{fig:ErdosPosa:Paths}), it follows by the choice of $d, h \geq K$ that adding $P_1$ to the set $U^1_{m_1} \cup U^2_{m_1}$, and $P_2$ to the set $U^2_{m_2} \cup U^3_{m_2}$ and $P_3$ to the set $U^3_{m_3} \cup U^1_{m_3}$ yields a $K$-fat model of $X$ in $G'$.
\end{proof}

\begin{lemma} \label{lem:ErdosPosa:NoThreeFatXs}
    Let $X$ be the graph from \cref{constr:X}, and let $h,d,m \in \N$. Then $G'_{h,d,m}$ does not contain three $3$-fat models of $X$ that have pairwise distance at least $3$ from each other.
\end{lemma}

\begin{proof}
    Let $h,d,m \in \N$ be given, and set $G' := G'_{h,d,m}$. Let $G$ be the (unique) subgraph of $G'$ that is isomorphic to $G_{h,d,m}$. We prove the following statement:
    \labtequtag{eq:ErdosPosa:FatX}
    {\emph{Every $3$-fat model of $X$ in $G'$ either contains $\RR(G')$, or it includes a path in $G$ between $\SSS(G)$ and $\TTT(G)$.}}
    {$\triangle$}
    By \cref{lem:NoTwoPaths}, this immediately implies the lemma.

    To prove \eqref{eq:ErdosPosa:FatX}, let $(\cU, \cE)$ be a $3$-fat model of $X$ in $G'$. Without loss of generality assume that $(\cU, \cE)$ is minimal, in that there is no $3$-fat model of $X$ in $G'$ whose vertex set is a proper subset of $Y:= (\bigcup_{x \in V(X)} U_x) \cup (\bigcup_{e \in E(X)} V(E_e))$.
    Since we are done if $(\cU, \cE)$ contains $\RR(G)$ or a path in $G$ between $\SSS(G)$ and $\TTT(G)$, we may assume that $Y$ contains neither $\RR(G)$ nor any such path. 

    We first show that each of the submodels of $(\cU, \cE)$ corresponding to some $M_i$ is essentially some $L_{i_j}$. More precisely, there are $i_1, i_2, i_3 \in [m]$ with $i_1<i_2<i_3$ such that $L_{i_1}, L_{i_2}, L_{i_3} \subseteq Y$, and such that each branch set $U_{m_1}, U_{m_2}, U_{m_3}$ contains precisely two of $s_{i_1}, s_{i_2}, s_{i_3}, t_{i_1}, t_{i_2}, t_{i_3}$ (but not both $s_{i_j}$ and $t_{i_j}$ for any $j \in [3]$).
    For this, let us note that since every $L_i$ can be separated from $G$ by removing two vertices ($s_i$ and $t_i$), and because the $(154\times 154)$-grid is `almost' $3$-connected (except for its four corners), the minimality condition on $(\cU, \cE)$ implies that for every $j \in [3]$, the submodel of $(\cU, \cE)$ corresponding to $M_j$ either meets each $L_i$ at most in some path, or it is almost contained in some $L_i$, except for its branch sets corresponding to $m_j$ and $m_{j-1 \pmod{3}}$ which meet $L_i$ but need not be included in it. 
    If the former case applies to one of $M_1, M_2, M_3$, then it follows that $\tilde{G} := \tilde{G}_{h, \max\{d,3K\} m}$ contains a $3$-fat model of $M_j$, where $\tilde{G}$ is the graph obtained from $G_{h,d,m}$ by adding, for every $i \in [m]$, a path of length $d+1$ between $t_i$ and $s_{m-i+1}$. But since $M_j$ is isomorphic to the $(154 \times 154)$-grid, this contradicts \cref{lem:NoFatGrids:EP}.  

    Hence, for every $j \in [3]$ there is some $i_j \in [m]$ such that the submodel of $(\cU,\cE)$ corresponding to $M_j$ is almost contained in $L_j$, except that the branch sets corresponding to $m_j$ and $m_{j-1 \pmod{3}}$ only meet $L_{i_j}$ but need not be included in $L_{i_j}$. In fact, by the structure of $X$, it follows that each branch set $U_{m_1}, U_{m_2}, U_{m_3}$ meets $G$, and in particular contains precisely two of $s_{i_1}, s_{i_2}, s_{i_3}, t_{i_1}, t_{i_2}, t_{i_3}$ (but not both $s_{i_j}$ and $t_{i_j}$ for any $j \in [3]$).

    Since branch sets are connected, it follows that there are paths $P_1 \subseteq G'[U_{m_1}]$, $P_2 \subseteq G'[U_{m_2}]$ and $P_3 \subseteq G'[U_{m_3}]$ whose endvertices are in $\{s_{i_1}, s_{i_2}, s_{i_3}, t_{i_1}, t_{i_2}, t_{i_3}\}$, but none of $P_{1}, P_2, P_3$ has both endvertices in $\{s_{i_j}, t_{i_j}\}$ for any $j \in [n]$. Since $L_{i_1}, L_{i_2}, L_{i_3}$ are `blocked' by the submodels of $(\cU, \cE)$ corresponding to $M_1, M_2, M_3$, respectively, $P_1, P_2, P_3$ are contained in $\bar{G} := G' - (V(L_{i_1} \cup L_{i_2} \cup L_{i_3}) \setminus V(G))$. 

    We need the following auxiliary statement:

    \begin{claim} \label{claim:ErdosPosa:NoTwoTPaths}
        Suppose there are $j \in \{1,3\}$ and $Q \in \{P_1, P_2, P_3\}$ such that $Q$ contains $s_{i_j}$ but not $s_{i_2}$. 
        If a subpath $Q'$ of $Q$ has endvertices $s_{\ell}, s_{\ell'} \in \SSS(G)$ such that $\ell < i_2 < \ell'$, then $Q' \not\subseteq G$.
    \end{claim}

    We remark that by the symmetry of $G$, \cref{claim:ErdosPosa:NoTwoTPaths} is also true with $\TTT(G)$ instead of $\SSS(G)$.
    \smallskip

    \begin{claimproof}
        Let $Q'$ be any such subpath of $Q$, and suppose for a contradiction that $Q' \subseteq G$. 
        Let further $W \in \{P_1, P_2, P_3\}$ be the path containing $s_{i_2}$, and note that $W \neq Q$ by the assumption on $Q$.
        
        Let us first observe that since $W \subseteq Y$, and $Y$ does not contain a path in $G$ between $\SSS(G)$ and $\TTT(G)$, and because $t_{i_2} \notin V(W)$ (as already $s_{i_2} \in V(W)$), it follows that $W$ meets $\SSS(G)$ in at least one other vertex than $s_{i_2}$. Let $W'$ be the (unique) subpath of $W$ that starts in $s_{i_2}$, ends in $\SSS(G)\setminus \{s_{i_2}\}$, and is internally disjoint from $\SSS(G)$. In particular, $W' \subseteq G$. Let $s_{k}$ be the other endvertex of $W'$ in $\SSS(G)$, and let $k^{\min} := \min\{i_2, k\}$ and $k^{\max} := \max\{i_2, k\}$.
        
        Now since $\ell < i_2 < \ell'$, it follows that $\ell, k^{\min} < \ell', k^{\max}$. Since $Q', W'$ are both paths in $G \cong G_{h,d,m}$ with one endvertex in $\{s_{k^{\min}}, s_{\ell}\}$ and the other in $\{s_{k^{\max}}, s_{\ell'}\}$, it follows by \cref{cor:NoTwoPathsBetweenSameAdhesion} that either $d_{G}(Q', W') \leq 2$, or one of $Q', W'$ contains $\RR(G) = \RR(G')$. Since $(\cU, \cE)$ is $3$-fat, the former cannot hold, and since $(\cU, \cE)$ does not contain $\RR(G)$ by assumption, neither the latter can hold (as $Q', W'$ are subpaths of distinct paths of $P_1, P_2, P_3$, and hence are contained in (distinct) branch sets of $(\cU,\cE)$), a contradiction.
    \end{claimproof}

    By the symmetry of $X$, we may assume without loss of generality that $s_{i_2} \in V(P_2)$. 
    Then at most one of $s_{i_1}, s_{i_3}$ is contained $P_2$, and hence at least one of them is contained in $P_1 \cup P_3$; we may assume that $s_{i_1} \in V(P_1)$, the other case is analogous.
    We can now use \cref{claim:ErdosPosa:NoTwoTPaths} to determine the other endvertex of $P_1$.

    \begin{claim} \label{claim:ErdosPosa:EndvertexOfP3}
        The other endvertex of $P_1$ is $t_{i_2}$.
    \end{claim}

    \begin{claimproof}
        Let $P'_1$ be a maximal subpath of $P_1$ which starts in $s_{i_1}$, is contained in $G$, and ends in $\SSS(G) \cup \TTT(G)$. Since $(\cU, \cE)$, and hence $P_1$, does not contain a path in $G$ between $\SSS(G)$ and $\TTT(G)$, the subpath $P'_1$ of $P_1$ ends in a vertex $s_j \in \SSS(G)$.
        Since $i_1 < i_2$, it follows by \cref{claim:ErdosPosa:NoTwoTPaths} that $j < i_2$.
        In particular, as $i_2 < i_3$, we have $j \neq i_3$. Hence, $P'_1 \subsetneq P_1$. As $P'_1$ was chosen maximal, it follows that $P_1$ goes through $L_{j}$ after it reached $s_j$, and hence (as $L_{j}$ is separated from $G$ by $s_j, t_j)$, the path $P_1$ eventually meets $t_j$.

        \begin{figure}[ht]
            \centering
            \begin{tikzpicture}[scale=0.49,auto=left]

\definecolor{dgreen}{rgb}{0,0.7,0}
\definecolor{dgrey}{rgb}{0.5,0.5,0.5}
\definecolor{lgrey}{rgb}{0.6,0.6,0.6}
\definecolor{llgrey}{rgb}{0.7,0.7,0.7}

\tikzstyle{every node}=[gray, inner sep=1.5pt, fill=gray,circle,draw]

\node[ForestGreen,ultra thick] (v2) at (2,0) {};
\node[magenta,ultra thick] (u2) at (2,3) {};
\node[blue,ultra thick] (y2) at (2.6,5.5) {};
\node[blue,ultra thick] (z2) at (3.2,7.5) {};

\node[blue,ultra thick] (v3) at (26,0) {};
\node[blue, ultra thick] (w3) at (26,2) {};
\node[magenta, ultra thick] (x3) at (25.8,4.5) {};
\node[ForestGreen, ultra thick] (z3) at (24.8,7.5) {};

\draw[blue,very thick] (z2) to [bend left=50] (y2);
\draw[blue,very thick,->] (y2) to [bend left=10] (1,5.9);
\draw[blue,very thick,->] (w3) to [bend left=10] (24.8,2.7);
\draw[blue,very thick] (27.7,1.9) to [bend left=10] (w3);

\draw[very thick] (3,8) to [bend right=20] (1.5,-0.5);
\draw [very thick] plot [smooth, tension=0.4] coordinates {(3,8) (14,11) (25,8)};
\draw[very thick] (25,8) to [bend left=20] (26.5,-0.5);
\draw[very thick] (1.5,-0.5) -- (26.5,-0.5);

\draw[magenta,very thick,->] (2,3) to [bend right=8] (4,2.6);

\draw[dgrey,very thick,dash pattern={on 10pt off 4pt}] (0.5,4) to [bend right=10] (10,3);
\draw[dgrey,very thick,dash pattern={on 10pt off 4pt}] (27.5,5.5) to [bend left=10] (18,4.5);

\draw[orange,very thick,dashed,->] (4,6) to [bend left=40] (3.2,1);
\draw[orange,very thick,dashed,->] (24.3,2.5) to [bend left=8] (4,1);

\draw[orange,very thick,dashed,->] (4.2,6.3) to [bend left=8] (24,7);
\draw[orange,very thick,dashed,->] (24.5,2.9) to [bend left=40] (24.1,6.3);

\tikzstyle{every node}=[]

\draw[decorate, decoration={brace, amplitude=1ex, raise=1ex}]
  (0.5, -0.5) -- (0.5, 4) node [pos=.5, left=2.5ex] {$S'$};
\draw[decorate, decoration={brace, amplitude=1ex, raise=1ex}]
  (0.5, 4.2) -- (0.5, 8) node [pos=.5, left=2.5ex] {$S$};

\draw[decorate, decoration={brace, amplitude=1ex, raise=1ex}]
  (27.5, 5.5) -- (27.5, -0.5) node [pos=.5, right=2.5ex] {$T$};
\draw[decorate, decoration={brace, amplitude=1ex, raise=1ex}]
  (27.5, 8) -- (27.5, 5.7) node [pos=.5, right=2.5ex] {$T'$};

\tikzstyle{every node}=[]
\draw[blue,right] (z2) node [] {$s_{i_1}$};
\draw[blue,below right] (y2) node [] {$s_{j}$};
\draw[magenta,below right] (1.9,2.8) node []  {$s_{i_2}$};
\draw[ForestGreen,above right] (2,0.1) node [] {$s_{i_3}$};

\draw[ForestGreen,below] (z3) node [] {$t_{i_3}$};
\draw[magenta,below] (26,4.3) node []  {$t_{i_2}$};
\draw[blue,above left] (26,-0.1) node [] {$t_{i_1}$};
\draw[blue,below left] (26.1,2) node [] {$t_{j}$};

\draw[blue] (1.5,6.6) node [] {$P_1$};

\draw[orange] (4.6,4) node [] {\Huge \ding{55}};
\draw[orange] (14,0.8) node [] {\Huge \ding{55}};
\draw[orange] (23.8,3.7) node [] {\Huge \ding{55}};
\draw[orange] (14,7.4) node [] {\Huge \ding{55}};

\draw (22,10) node [] {\large $G$};

\end{tikzpicture}
            \caption{The blue path $P_1$ starting from $s_{i_1}$ can `jump' back and forth between $S := \{s_1, \dots, s_{i_2-1}\} \subseteq \SSS(G)$ and $T:= \{t_{i_2}, \dots, t_m\} \subseteq \TTT(G)$, but it cannot contain any vertices of $S' := \SSS(G) \setminus S$ and $T' := \TTT(G) \setminus T$: 
            any subpath of $P_1$ from $S$ to $S'$ that is contained in $G$ contradicts \cref{claim:ErdosPosa:NoTwoTPaths}, and
            any subpath of $P_1$ from $\TTT(G) \supseteq T$ to $S'$ that is contained in $G$ contradicts the assumption on $P_1$. (Similar for $T'$). Hence, $s_{i_3}, t_{i_3} \notin V(P_1)$, and thus $P_1$ has to end in $t_{i_2}$.}
            \label{fig:ErdosPosa:PathP1}
        \end{figure}
        Now if $P_1$ does not contain $t_{i_2}$, then, since $j < i_2$, \cref{claim:ErdosPosa:NoTwoTPaths} (applied to $\TTT(G)$ instead of $\SSS(G)$) implies that $P_1$ does not contain a subpath $P'_1$ with endvertices $t_{k}, t_{k'} \in \TTT(G)$ such that $P'_1 \subseteq G$, and such that $k < i_2 < k'$.
        But since $i_3 > i_2$, and as $P_1$ does not contain a subpath in $G$ between $\SSS(G)$ and $\TTT(G)$, it follows that $P_1$ contains neither $s_{i_3}$ nor $t_{i_3}$ (cf.\ \cref{fig:ErdosPosa:PathP1}). But since it neither contains $s_{i_2}, t_{i_2}$ by assumption, nor $t_{i_1}$ (as it contains $s_{i_1}$), this is a contradiction. 
    \end{claimproof}

    \begin{claim}
        The other endvertex of $P_2$ is $t_{i_1}$.
    \end{claim}

    \begin{claimproof}
        Let $j \in \{2,3\}$ such that $t_{i_1} \in V(P_j)$. By the symmetry of $G$ and since $t_{i_2} \in V(P_1)$, and hence $t_{i_2} \notin V(P_j)$, the same argument as in \cref{claim:ErdosPosa:EndvertexOfP3} shows that $s_{i_2} \in V(P_j)$. As $s_{i_2} \in V(P_2)$ by assumption, this implies that $P_2$ ends in $t_{i_1}$.
    \end{claimproof}

    Hence, $P_1$ has endvertices $s_{i_1}, t_{i_2}$ and $P_2$ has endvertices $s_{i_2}$ and $t_{i_1}$. But then $P_3$ must have endvertices $s_{i_3}, t_{i_3}$, a contradiction because $P_3 \subseteq U_{m_3}$ but $s_{i_3}, t_{i_3}$ are opposite corners of $L_{i_3}$, and are hence contained in distinct branch sets of $(\cU,\cE)$.
\end{proof}

We are now ready to prove \cref{main:ErdosPosaPlanar:copy}:

\begin{proof}[Proof of \cref{main:ErdosPosaPlanar:copy}]
    Let $X$ be the graph from \cref{constr:X}. Given $K,D,n \in \N$, we let $G := G_{h,d,m}$ as provided by \cref{lem:ErdosPosa:NoSmallHittingSet}. Then \ref{itm:ErdosPosa:NoSmallHittingSet} follows from \cref{lem:ErdosPosa:NoSmallHittingSet} and \ref{itm:ErdosPosa:NoThreeFatXs} follows from \cref{lem:ErdosPosa:NoThreeFatXs}.
\end{proof}

\section{Concluding remarks}\label{sec:con}

One motivation for the (weak) fat minor conjecture was that it would of allowed for the graph minor structure theorem of Robertson and Seymour \cite{robertson2003graph} to instantly be lifted via quasi-isometry to the coarse setting where some fat minor is forbidden instead of just a minor.
Unfortunately, \Cref{main:Counterexample} shows that this is not possible. Of course, this does not rule out the possibility of some kind of coarse structure theorem for graphs forbidding a fat minor being possible, it just needs to be more complicated than was hoped.
Possibly the construction in \Cref{main:Counterexample} can provide some clue as to what this more complicated structure could be.
In this section we discuss some weakening of \Cref{conj:CoarseGrid} that might still hold.

\Cref{conj:CoarseGrid} might still hold for 2-fat minors (which in this case is essentially the same as induced minors).

\begin{conjecture}
    Let $n \in \mathbb{N}$.
    Then, there exist some $M,A,g\in \mathbb{N}$ such that every graph with no $(n\times n)$-grid induced minor is $(M,A)$-quasi-isometric to a graph of tree-width at most $g$.
\end{conjecture}

A positive result in this direction is a theorem of Korhonen \cite{korhonen2023grid} that graphs of bounded degree forbidding a grid as an induced minor have bounded tree-width.
Possibly this could be strengthened to fat minors.

\begin{conjecture}
    Let $K,d,n \in \mathbb{N}$.
    Then, there exists some $g\in \mathbb{N}$ such that every graph with maximum degree at most $d$ and with no $K$-fat $(n\times n)$-grid minor has tree-width at most $g$.
\end{conjecture}

Note that for graphs with bounded degree, having bounded tree-width is (quantitively) equivalent to being quasi-isometric to a graph of bounded tree-width due to results on coarse tree-width \cite{hickingbotham2025graphs,nguyen2025coarse}.

The counterexample \cite{CounterexAgelosPanosConjecture} to the fat minor conjecture \cite{GP23} has bounded degree.
Possibly the bounded degree version of the weak fat minor conjecture \cite{CounterexAgelosPanosConjecture} could still hold.

\begin{conjecture}
    Let $K,d \in \mathbb{N}$, and let $H$ be a graph.
    Then, there exist some $M,A\in \mathbb{N}$ and a graph $H'$ such that every graph with maximum degree at most $d$ and with no $K$-fat $H$ minor is $(M,A)$-quasi-isometric to a graph with no $H'$ minor. 
\end{conjecture}

Although \cref{conj:CoarseGrid} is false, there can still be witnesses other than large grid fat minors to a graph not having a quasi-isometry to a graph of bounded tree-width.
Many proofs of the grid theorem \cite{robertson1986graph} proceed by first showing that a graph having large tree-width is witnessed instead by some structure with suitably high connectivity properties. Grids can then be found as the canonical witness for large tree-width within such structures with Ramsey-type arguments that make use of Menger's theorem.
Since the (weak) coarse Menger conjecture is false \cite{nguyen2025counterexample,NewCounterexToCoarseMenger}, the Menger's step in the strategy certainly breaks down in the coarse setting.
However, the first half of showing that large tree-width is witnessed by some structure with suitably high connectivity properties is interesting in its own right, and perhaps there could still be coarse analogues of this.
We propose one such coarse analogue.

\begin{conjecture}\label{conj:weakgrid}
    Let $K,n,m,r \in \mathbb{N}$.
    Then, there exist some $M,A,g\in \mathbb{N}$ such that every graph with no $(M,A)$-quasi-isometry to a graph of tree-width at most $g$ contains connected sets $U_1,\ldots,U_n$ that are pairwise at least $K$ apart and such that for every pair $1 \le i<j\le n$, there exist no $m$ balls of radius at most~$r$ hitting all paths between $U_i$ and $U_j$.
\end{conjecture}

\noindent Note that if a graph contains a large grid as a $D$-fat minor for large enough $D$, then it contains such connected sets $U_1, \ldots U_n$, and therefore \cref{conj:weakgrid} is a weakening of \cref{conj:CoarseGrid}.
Moreover, note that the (qualitative) converse of \cref{conj:weakgrid} is straightforward to check.\footnote{If $K$ is large compared to $M,A$, then for every graph $G$ with such sets $U_i$ and every $(M,A)$-quasi-isometry $f$ from $G$ to some graph $H$, we can find connected, disjoint sets $W_i$ in $V(H)$, each containing $f(U_i)$, by \cref{lem:QIPreservesConn}. If also $r$ is large compared to $M,A$, then applying Menger's theorem in $H$ yields between any two sets $W_i,W_j$ at least $m$ disjoint paths $P^{ij}_k$. Now the sets $W_i, V(P^{ij}_k)$ can easily be turned into a bramble~$\cB$ in $H$ of order at least $\ell := \min\{m,n\}$, witnessing that $H$ has tree-width at least $\ell-1$ \cite{ST1993GraphSearching} (for every $(i_0, \dots, i_{n}) \in [n] \times [m]^n$, put the set $W_{i_0} \cup \bigcup_{j \in [n]} V(P^{i_0j}_{i_j})$ in $\cB$).}
\medskip

If \cref{conj:weakgrid} is true, then it might be possible that \cref{conj:weakgrid} could be further refined to more precisely give the structural obstructions to being quasi-isometric to a graph of bounded tree-width.
Of course, in light of the counterexample in this paper, such a list of obstructions is likely very complicated.
Towards this it would certainly be helpful to know exactly what obstructions causes the (weak) coarse Menger conjecture \cite{albrechtsen2024menger,GP23,nguyen2025counterexample} to be false.
Nguyen, Scott, and Seymour have very recently shown that the coarse Menger conjecture \cite{albrechtsen2024menger,GP23} is true for graphs of bounded path-width \cite{AS5} and for `paths across a disc' \cite{AS6}.

The coarse Menger conjecture \cite{albrechtsen2024menger,GP23} is false for graphs of bounded tree-width since Nguyen, Scott, and Seymour's \cite{nguyen2025counterexample} construction has bounded tree-width \cite{CounterexAgelosPanosConjecture}.
However, using the Erd\H{o}s-P\'{o}sa property for subtrees \cite{gyarfas1970helly}, it is also possible to prove the weak coarse Menger conjecture~\cite{nguyen2025counterexample} for graphs of bounded tree-width.
Since the property that class of graphs satisfies the (weak) coarse Menger conjecture is preserved by quasi-isometry, it is possible to show that the construction in the present paper is a counterexample to \cref{conj:CoarseGrid} using this fact (as it is also a counterexample to the weak coarse Menger conjecture by \cref{lem:NoTwoPaths,lem:NoSmallSeparator}).
This avoids needing to prove \cref{lem:KnMinors} to disprove \cref{conj:CoarseGrid}, but of course this lemma is still needed to prove \cref{main:Counterexample} in full and show that the construction is furthermore a counterexample to the weak fat minor conjecture \cite{CounterexAgelosPanosConjecture}.
\medskip

Although \cref{main:ErdosPosaPlanar} shows that not all planar graphs have the coarse Erd\H{o}s-P\'{o}sa property, curiously it still leaves open the case of having two copies.

\begin{conjecture}
    For every planar graph $H$ there exists a constant $c$ such that he following holds for every graph $G$.
    If $G$ does not contain two $K$-fat minor-models of $H$ that are at distance at least $K$ from each other, then there is a set $Z$ of at most $c$ vertices of $G$ such that the ball~$B_{G}(Z,cK)$ of radius $cK$ around $Z$ in $G$ meets all $K$-fat $H$ minors models in $G$.
\end{conjecture}

There are other related variants of the Erd\H{o}s-P\'{o}sa property.
For instance, Ahn, Gollin, Huynh, and Kwon~\cite{AGHKInducedErdosPosa} proved an induced Erd\H{o}s-P\'{o}sa theorem and Dujmović, Joret, Micek, and Morin~\cite{DJMMErdosPosalong} showed the Erd\H{o}s-P\'{o}sa property for far apart cycles.
While planar graphs do not have the coarse Erd\H{o}s-P\'{o}sa property, we conjecture a strengthening of Dujmović, Joret, Micek, and Morin's theorem for planar graphs as follows.

\begin{conjecture}
    For every planar graph $H$ there is a function $f : \mathbb{N} \to \mathbb{N}$ and a constant $c$ such that the following holds for every graph $G$.
    If $G$ does not contain $n$ disjoint $H$ minor-models that have pairwise distance at least $d$ from each other, then  there is a set $Z$ of at most $f(n)$ vertices of $G$ such that
    $G-B_{G}(Z,cd)$ contains no $H$ minor.
\end{conjecture}

\section*{Acknowledgements}

We thank Raphael Jacobs for suggesting to consider $K$-fat $n$-path-connected sets in the proof of \cref{lem:NoFatGrids}.
We also thank Robert Hickingbotham for helpful comments on the paper, and in particular for pointing out that the weak coarse Menger conjecture holds for graphs of bounded tree-width.
We also thank Marthe Bonamy and Romain Bourneuf for helpful discussions leading us to \Cref{main:ErdosPosaPlanar}.

\printbibliography

\end{document}